\documentclass[11pt,a4paper]{article}
\usepackage{amsmath,amsthm,amssymb,amsfonts,color,graphicx}

\newcommand{\C}{{\mathbb{C}}}          % \C       = complexos
\newcommand{\N}{{\mathbb{N}}}          % \N       = naturais
\newcommand{\Proj}{{\mathbb{P}}}        % \Pro     = projectivo
\newcommand{\R}{{\mathbb{R}}}          % \R       = reais
\newcommand{\p}{{\mathfrak{p}}}       %

\newcommand{\sol}{{\mathfrak{o}}}

\newcommand{\hol}{{\mathfrak{hol}}}
\newcommand{\XIS}{{\mathfrak{X}}}

\newcommand{\rr}{\rightarrow}
\newcommand{\lrr}{\longrightarrow}

\newcommand{\calH}{{\cal H}}             %
\newcommand{\calL}{{\cal L}}             %
\newcommand{\calR}{{{\cal R}^\xi}}             %
\newcommand{\calV}{{\cal V}}             %

\newcommand{\gdois}{{\mathrm{G}_2}}

\newcommand{\SO}{{\mathrm{SO}}}
\newcommand{\SU}{{\mathrm{SU}}}

\newcommand{\na}{{\nabla}}

\newcommand{\tr}[1]{{\mathrm{tr}}\,{#1}}

\newcommand{\ric}{{\mathrm{ric}}}
\newcommand{\Scal}{{\mathrm{Scal}}}
\newcommand{\dx}{{\mathrm{d}}}

\newcommand{\inv}[1]{{#1}^{-1}}
\newcommand{\papa}[2]{\frac{\partial#1}{\partial#2}}

\newcommand{\gammadot}{{\dot{\gamma}}}
\newcommand{\gammadotdot}{{\ddot{\gamma}}}
\newcommand{\ydot}{{\dot{y}}}
\newcommand{\ydotdot}{{\ddot{y}}}

\newcommand{\ctildezero}{{\tilde{c}_0}}
\newcommand{\ctildeum}{{\tilde{c}_1}}

\newcommand{\mg}{{\mathrm{g}}}

\newtheorem{teo}{Theorem}[section]

\newtheorem{coro}{Corollary}[section]
\newtheorem{prop}{Proposition}[section]

\newenvironment{rema}[1][Remark.]{\begin{trivlist}
\item[\hskip \labelsep {\bfseries #1}]}{\end{trivlist}}

\def\cyclic{\mathop{\kern0.9ex{{+}
\kern-2.2ex\raise-.28ex\hbox{\Large\hbox{$\circlearrowright$}}}}\limits}

\title{On vector bundle manifolds with spherically 
\\
symmetric metrics}

\author{R. Albuquerque}

\begin{document}

%\begin{color}{DarkBlue}
%\begin{color}{black}

\maketitle

%\date{\today} %{1 dezembro 2009}

\markright{\sl\hfill  R. Albuquerque\hfill}

\begin{abstract}

We give a general description of the construction of weighted spherically symmetric metrics on vector bundle manifolds, i.e. the total space of a vector bundle $E\lrr M$, over a Riemannian manifold $M$, when $E$ is endowed with a metric connection. The tangent bundle of $E$ admits a canonical decomposition and thus it is possible to define an interesting class of two-weights metrics with the weight functions depending on the fibre norm of $E$; hence the generalized concept of spherically symmetric metrics. We study its main properties and curvature equations. Finally we focus on a few applications and compute the holonomy of Bryant-Salamon type $\gdois$ manifolds.

\end{abstract}

\vspace*{4mm}

\vspace*{10mm}

\section*{Introduction}

We start this introduction by recalling three lines of independent research to which the present article intends to be related.

Important geometry has appeared in the last decades regarding the tangent bundle of a Riemannian manifold endowed with the metric found by S. Sasaki. As it is well-known, given a Riemannian manifold $M$, the metric on the total space of $\pi:TM\lrr M$ is defined by the canonical splitting of the tangent bundle of $TM$, via the Levi-Civita connection. The bundle projection $\pi$ becomes a Riemannian submersion. But one may also consider so-called weighted Sasaki metrics, which show even further intricate features, and then no canonical equations apply. Finding the curvature involves many computations, which are a step towards the deduction of the holonomy algebra.

In the geometry of complex and holomorphic Hermitian vector bundles $\pi:E\lrr M$ many interesting developments have resulted from the introduction, by S. Kobayashi and many others, of some special functions which vary with the fibre norm. It is quite interesting to observe the role of the zero-section in many theorems in \cite{Koba1}. If $M$ is K\"ahler, then $E$ becomes a Hermitian manifold in a natural way and also with freely chosen weights following the natural decomposition of $TE$.

In the search of real exceptional geometries, R. Bryant and S. Salamon have discovered complete $\gdois$ holonomy metrics on the total space of $E=\Lambda^2_-T^*M\lrr M$ where $M=S^4$ or $\C\Proj^2$. They used the natural decomposition of the tangent bundle of $E$ and two weight functions carefully chosen.

Now, on any vector bundle $\pi:E\lrr M$ over a Riemannian manifold $M$, endowed with a Riemannian metric structure, that is, a smooth section $g_{_E}$ of $S^2E^*$ non-degenerate and positive definite, and endowed with a compatible connection ${D}^{^E}$, one can equip the total space of $E$ with a Riemannian metric in the usual fashion. One may see this idea in general in few but varied contexts, of which the three themes above are example. Also we may refer the reader to \cite{BeleWei,BenLoubWood1,MurWals,Tapp} for pertinent aspects of such geometries, different from those addressed here. To the best of our knowledge, a theory considering metrics with weights is lacking in the literature.

In this article we study a subcase which is both general and most natural. Namely, the construction of Riemannian metrics $\mg_{_{M,E}}$ on the total space of $E\lrr M$ with the weights smoothly depending on the fibre squared norm. Our techniques have developed from our previous studies on fibre bundles, which in turn stem from the methods in \cite{Obri}. We shall be able to write $TE\simeq\pi^*TM\oplus{\pi^\star}E$, so we remark that throughout we need to distinguish the vertical lifts from the horizontal, and also to be permanently aware of the two pull-back structures $\pi^*TM$ and ${\pi^\star}E$, both being nothing else but the common pull-back by $\pi$.

In the second chapter we give several results on the curvature of the metric $\mg_{_{M,E}}$ with weights. In particular a computation which helps on the search for the global holonomy of the Riemannian manifold $E$ without actually determining it. Yet, how far are we still from knowing the global holonomy is a question that remains. Among our curvature results we give a criteria for Einsten metrics.

In the third chapter we give applications to Hermitian geometry, in the type of Sasaki type metrics and almost complex structures. Then we bring again to the front the $\gdois$ spaces of Bryant-Salamon. The results we started proving in \cite{Alb2014b} are finally taken to an end. They justify completely the previous efforts and the further particular details needed in the proof of a last theorem, which regards the real and complex hyperbolic spaces $\calH^4$ and $\calH_\C^2$, or any other negatively curved Einstein anti-self-dual manifold $M$. The theorem says that the disk bundles $D_{r_0,+}M$ contained in $\Lambda^2_+T^*M$ with essentially the Bryant-Salamon integrable $\gdois$ metric $\mg_\phi$, have indeed $\gdois$ holonomy. We also compute the holonomy of certain $\gdois$ metrics on $\Lambda^2_\pm T^*M$ for any K3 surface $M$ with the canonical metric.

The author is very grateful for the remarks and careful reading of the manuscript by an anonymous referee.

\section{A natural Riemannian structure}

\subsection{The metric on the manifold $E$ and its Levi-Civita connection}
\label{TmotmEaiLCc}

Let $M$ be a Riemannian manifold and $g_{_M}$ denote its metric tensor. Let $\pi:E\lrr M$ be a rank $k$ vector bundle over $M$. Our study considers as its main subject the manifold $E$. We also assume such vector bundle is endowed with a metric $g_{_E}$ and a compatible connection ${D}^{^E}$. A first equation is thus ${D}^{^E}g_{_E}=0$. The manifold $E$ has a canonical atlas of trivializations. Then the fibres $E_{\pi(e)}=\inv{\pi}(\pi(e))$, for each $e\in E$, have the natural structure of submanifolds, with tangent bundle the trivial bundle. Moreover we have an exact sequence $0\lrr\calV\lrr TE\stackrel{\dx\pi}{\lrr}\pi^* TM\lrr0$ of vector bundles over the \textit{manifold} $E$ and, by construction, the vertical or kernel bundle $\calV\lrr E$ identifies with ${\pi^\star}E\lrr E$. We then use the connection ${D}^{^E}$ to induce a splitting of $TE$ as $\calH^{{D}^{^E}}\oplus\calV$. Since $\calH^{{D}^{^E}}$ is canonically identified to the vector bundle $\pi^*TM$, through the restriction of the map $\dx\pi$, we may finally write
\begin{equation}\label{canonicaldecompo}
 \calH^{{D}^{^E}}\oplus\calV=TE\simeq \pi^*TM\oplus{\pi^\star}E\ .
\end{equation}
This canonical decomposition of the tangent bundle of $E$ has even further virtues. The terms horizontal and vertical used for the components of any tangent vector $X=X^h+X^v$ at each point $e\in E$ are defined accordingly. We also have a natural vector field $\xi$, a tautological section of vertical directions, defined by $\xi_e=e\in{\pi^\star}E$. The important role played by $\xi$ is shown through a projection onto $\calV$ with kernel $\calH^{{D}^{^E}}$:
\begin{equation}\label{projeccaoprincipal}
 {\pi^\star}{D}^{^E}_X\xi=X^v\ .
\end{equation}
To see this quickly, we may take a frame $(e_1,\ldots,e_k)$ of $E$ on an open set $U\subset M$. Then any point $e\in\inv{\pi}(U)\subset E$ is written uniquely as $e=\sum_{\alpha=1}^k y^\alpha e_\alpha$ and a vertical tangent vector to $E$ is written as $X=\sum_\beta x^\beta{\pi^\star}e_\beta$ ($y^\alpha,x^\beta\in\R$). Then $\dx\pi(X)=0$ and 
\[ {\pi^\star}{D}^{^E}_X\xi=\sum_{\alpha=1}^k\bigl(\dx y^\alpha(X){\pi^\star}e_\alpha+y^\alpha{\pi^\star}{D}^{^E}_X{\pi^\star}e_\alpha\bigr)
=\sum_{\alpha=1}^k x^\alpha{\pi^\star}e_\alpha=X  \ .   \]

On a chart of $M$ compatible with a trivialization of $E$ we can deduce the coordinate equations for a vector field in $\calH^{{D}^{^E}}$. On the other hand, we easily lift vector fields $X$ on $M$ to sections $X^h$ of $\pi^*TM$. These, given purely by differential geometry, are called the horizontal vector fields (the theory is indeed coherent, but the reader may see it with further detail in section \ref{somepropertiessection}).

Finally we introduce the metric structures in context. Clearly the manifold $E$ inherits a Riemannian structure $\pi^*g_{_M}\oplus{\pi^\star}g_{_E}$.
Letting $\na^{^M}$ denote the Levi-Civita connection of $M$, it is also clear that the connection ${D^{**}}=\pi^*\na^{^M}\oplus{\pi^\star}{D}^{^E}$ is a metric connection, i.e. ${D^{**}}(\pi^*g_{_M}\oplus{\pi^\star}g_{_E})=0$. Its torsion satisfies (throughout this article $X,Y,Z,W$ denote vector fields on the manifold $E$, if not stated differently elsewhere)
\begin{equation}\label{equacoesdatorsao}
\begin{cases}
  \dx\pi(T^{{D^{**}}}(X,Y))=\dx\pi ({D^{**}}_XY)-\dx\pi ({D^{**}}_YX)-\dx\pi([X,Y])=T^{\na^{^M}}(\dx\pi X,\dx\pi Y)=0 \\
   (T^{{D^{**}}}(X,Y))^v={\pi^\star}{D}^{^E}_XY^v-{\pi^\star}{D}^{^E}_YX^v-[X,Y]^v={\pi^\star}R^E(X,Y)\xi
\end{cases}.
\end{equation}
These two formulas are of the utmost importance and simple to prove; also they are similar to those found in \cite{Obri}. Recall that $R^E(X,Y)$ coupled with the metric of $E$ takes values in $\Lambda^2E^*$. Moreover, first by tensoriality and second by the previous formula, we have
\begin{eqnarray}
  {\pi^\star}R^E(X,Y)\xi &=& R^{{\pi^\star}D^{^E}}(X,Y)\xi   \nonumber \\ &=& {\pi^\star}R^E(X^h,Y^h)\xi  \nonumber \\ &=& 
  -[X^h,Y^h]^v\ .
\end{eqnarray}

Now we are much more interested in another metric defined on the manifold $E$ --- it is a metric arising as the above but with certain weight functions. First we consider the scalar function $r=\|\xi\|^2_{_E}$ defined on $E$, i.e. the squared radial-distance to the 0 section. Since $\xi$ is vertical, again by \eqref{projeccaoprincipal} we have
\begin{equation}\label{derivativeofr}
 \dx r=2\xi^\flat\ .
\end{equation}

The Riemannian manifold $E$ we wish to study in this article is defined by the metric
\begin{equation}\label{theweightedSasakimetric}
 \mg_{_{M,E}}=e^{2\varphi_1}\pi^*g_{_M}\oplus e^{2\varphi_2}{\pi^\star}g_{_E}
\end{equation}
where $\varphi_1,\varphi_2$ are smooth scalar functions on $E$ dependent only of $r$ and smooth at $r=0$ on the right, i.e. we assume $\varphi_i,\varphi_i',\varphi_i''\ldots$, $i=1,2$, exist and are continuous at 0. Notice we continue to assume \eqref{canonicaldecompo} implicitly. The map $\pi$ becomes a Riemannian submersion if and only if $\varphi_1=0$.
\begin{rema}
With $\varphi_1$ and any other functions of $r$ we use the notation 
$\varphi_1'=\papa{\varphi_1}{r}$.
\end{rema}

Clearly $D^{**}_X(e^{2\varphi_1}\pi^*g_{_M})=4\varphi_1'e^{2\varphi_1}\xi^\flat(X)\,\pi^*g_{_M}$, which is not so important but gives some clues to the following. Though it is quite easy to find metric connections for each summand of $\mg_{_{M,E}}$, we wish primarily to find a linear metric connection over $E$ keeping the same torsion of ${D^{**}}$. Hence we are led to consider ${\widetilde{D}}={D^{**}}+C$ with $C\in\Omega^0(S^2(T^*E)\otimes TE)$ given by ($a,b,c_1,c_2\in\Omega^0$ smooth functions of $r$)
\begin{equation}\label{formuladeC}
\begin{split}
 & C_XY\:=\: a\bigl(\xi^\flat(X)Y^h+\xi^\flat(Y)X^h\bigr)+\\
 &\ \ \qquad\qquad\qquad+c_1\langle X,Y\rangle_{_M}\xi+c_2\langle X,Y\rangle_{_E}\xi
 +b\bigl(\xi^\flat(X)Y^v+\xi^\flat(Y)X^v\bigr)\ .
 \end{split}
\end{equation}
We are using a common notation with brackets with obvious meaning and indexed when necessary. For instance, the notation $\langle X,Y\rangle_{_M}$ stands for $\pi^*g_{_M}(X^h,Y^h)$. Simple computations lead us to the following result, but along the following pages we shall continue with the $a,b,c$s above.
\begin{teo}\label{TeoremaformuladeCcoeficientes}
The linear connection ${\widetilde{D}}$ on the Riemannian manifold $E$ is a metric connection (${\widetilde{D}}\,\mg_{_{M,E}}=0$) if and only if
\begin{equation}\label{formuladeCcoeficientes}
 \begin{split}
   a=2\varphi_1'\qquad&\qquad c_1=-2\varphi_1'e^{2(\varphi_1-\varphi_2)}\\
 b=2\varphi_2'\qquad&\qquad c_2=-2\varphi_2'
 \end{split}\ .
\end{equation}
\end{teo}
\begin{proof}
 Solving the equation ${\widetilde{D}}_X\mg_{_{M,E}}\,(Y,Z)=0$ is equivalent to 
 \begin{eqnarray*}
& 0\,=\, -4\varphi_1'e^{2\varphi_1}\xi^\flat(X)\langle Y,Z\rangle_{_M} -4\varphi_2'e^{2\varphi_2}\xi^\flat(X)\langle Y,Z\rangle_{_E}+ e^{2\varphi_1}\langle C_XY,Z\rangle_{_M}+\\
& +e^{2\varphi_1}\langle Y,C_XZ\rangle_{_M}+ e^{2\varphi_2}\langle C_XY,Z\rangle_{_E}+e^{2\varphi_2}\langle Y,C_XZ\rangle_{_E} =
\\ & = -4\varphi_1'e^{2\varphi_1}\xi^\flat(X)\langle Y,Z\rangle_{_M}+ e^{2\varphi_1}a(\xi^\flat(X)\langle Y,Z\rangle_{_M}+ \xi^\flat(Y)\langle X,Z\rangle_{_M}) + \\ & +e^{2\varphi_1}a(\xi^\flat(X)\langle Y,Z\rangle_{_M}
+ \xi^\flat(Z)\langle X,Y\rangle_{_M}) 
-4\varphi_2'e^{2\varphi_2}\xi^\flat(X)\langle Y,Z\rangle_{_E}+ e^{2\varphi_2}(c_1\langle X,Y\rangle_{_M}\xi^\flat(Z)+\\ &
 +c_2\langle X,Y\rangle_{_E}\xi^\flat(Z)+ c_1\langle X,Z\rangle_{_M}\xi^\flat(Y)+ c_2\langle X,Z\rangle_{_E}\xi^\flat(Y))+\\ &
 +e^{2\varphi_2}b(2\xi^\flat(X)\langle Y,Z\rangle_{_E}+\xi^\flat(Y)\langle X,Z\rangle_{_E}+ \xi^\flat(Z)\langle X,Y\rangle_{_E})\ .
 \end{eqnarray*}
Since this is valid for all vectors, we find six equations which then yield (\ref{formuladeCcoeficientes}).
\end{proof}
Since $C$ is symmetric, we still have $T^{{\widetilde{D}}}=T^{{D^{**}}}={\pi^\star}R^E(\ ,\ )\xi$. We are also going to abbreviate the notation for this last $\calV$-valued tensor: we let $\calR={\pi^\star}R^E(\ ,\ )\xi$.

Finally, the Levi-Civita connection of the metric $\mg_{_{M,E}}$ is the connection $\na^{^{M,E}}$ given by
\begin{equation}\label{theLeviCivitaconnection}
  \na^{^{M,E}}_XY=D^{**}_XY+C_XY+A_XY-\frac{1}{2}\calR(X,Y)
\end{equation}
with $C$ defined in (\ref{formuladeC},\ref{formuladeCcoeficientes}) and the $\pi^*TM$-valued 2-tensor $A$ defined by
\begin{equation}
 e^{2\varphi_1}\langle A(X,Y),Z\rangle_{_M}=\frac{e^{2\varphi_2}}{2}
 \bigl(\langle\calR(X,Z),Y\rangle_{_E}+\langle\calR(Y,Z),X\rangle_{_E}\bigr)\ .
\end{equation}
Notice $A$ is symmetric, so now we have $T^{\na^{^{M,E}}}=0$. On the other hand, since ${\widetilde{D}}=D^{**}+C$ is a metric connection, we just have to verify, which is very easy, that
\begin{equation}
 \mg_{_{M,E}}(A(X,Y)-\frac{1}{2}\calR(X,Y),Z) 
 =-\mg_{_{M,E}}(Y,A(X,Z)-\frac{1}{2}\calR(X,Z)) \ . 
\end{equation}
We note the formula $A(X,Y)=A_XY=A(X^h,Y^v)+A(X^v,Y^h)$ and stress that $A$ takes values in the horizontal distribution.

\subsection{Parallel vector fields and isometries of $\mg_{_{M,E}}$}
\label{somepropertiessection}

A first problem we would like to discuss regards the description of the parallel vector fields of $E$. Let us deduce their equations on a trivializing subset.

We let $x=(x^1,\ldots,x^m)$ be a chart of the base defined on an open subset $U\subset M$ ($\dim M=m$). If necessary restricting to a smaller open subset, we may take an orthonormal frame $\{e_1,\ldots,e_k\}$ of $E$ on $U$. Hence we have a trivialization $\inv{\pi}(U)\simeq U\times\R^k$ with coordinates $(x,y)$, linear on the fibres by assumption. Since any point $e\in\inv{\pi}(x)$ may be written as $e=\sum_\alpha y^\alpha e_\alpha$, the tautological vector field $\xi$ satisfies $\xi_e=\sum_\alpha y^\alpha{\pi^\star}e_\alpha$. We have $r=\sum_\alpha(y^\alpha)^2$ and denote $g_{_M}(\partial_i, \partial_j)=g_{ij}$, where the vectors $\partial_i=\papa{ }{x^i}$ denote the duals of the $\dx x^j$. This has inverse matrix $g^{jq}$. We also let $\pi^*\partial_i$ denote the lift of $\partial_i$ to the \textit{horizontal} part of $TE$. Below, Kronecker and Christoffel symbols have usual expression. The second are defined by $\na^M_{\partial_i}\partial_j=\Gamma_{ij}^{M,h}\partial_h$ and $D^{E}_{\partial_i}e_\alpha=\Gamma_{i\alpha}^{E,\beta}e_\beta$. Throughout indices satisfy $1\leq i,j,q,l\leq m$ and $1\leq\alpha,\beta,\epsilon\leq k$, and Einstein summation convention is assumed. For the curvature tensor we denote $R_{\beta\alpha ij}^E=\langle R^E(\partial_i,\partial_j)e_\alpha,e_\beta\rangle_{_E}$.

Note that the $\partial_i={\partial_i}_{(x,y)}$ also make sense in $\inv{\pi}(U)$, but such vector fields are not horizontal in general. It is easy to see, applying \eqref{projeccaoprincipal}, that
\begin{equation}\label{decompositionofchartvectorfields}
 \pi^*\partial_i=\partial_i-y^\alpha\Gamma_{i\alpha}^{E,\beta}{\pi^\star}e_\beta\ .
\end{equation}
Notice $\pi^*g_{_M}(\pi^*\partial_i,\pi^*\partial_j)=g_{ij}$ and ${\pi^\star}g_{{_E}}({\pi^\star}e_\alpha,{\pi^\star}e_\beta)=\delta_\alpha^\beta$. Henceforth, by \eqref{decompositionofchartvectorfields}, we find $\mg_{_{M,E}}(\partial_i,\partial_j)=e^{2\varphi_1}g_{ij}+e^{2\varphi_2}y^\alpha y^\gamma\Gamma_{i\alpha}^{E,\beta}\Gamma_{j\gamma}^{E,\beta}$.
It is interesting to observe that we do not have to make use of \eqref{decompositionofchartvectorfields} in the next deductions.

Following the orthogonal decomposition \eqref{canonicaldecompo}, any vector field on $E$ is written as $Y=Y^j\pi^*\partial_j+B^\alpha{\pi^\star}e_\alpha$. Then we may develop four equations for $\na^{^{M,E}}Y$ of different kind:
\begin{eqnarray}\label{equationofparallel}
  (\na^{^{M,E}}_{\pi^*\partial_i}Y)^q &=&   \papa{Y^q}{x^i}+Y^l\Gamma_{il}^{M,q}+ay^\alpha B^\alpha\delta_i^q+
  \frac{e^{2(\varphi_2-\varphi_1)}}{2}y^\alpha B^\beta R_{\beta\alpha ij}^Eg^{jq}  
  \label{equationofparallel1}\\
  (\na^{^{M,E}}_{{\pi^\star}e_\beta}Y)^q &=& \papa{Y^q}{y^\beta}+ay^\beta Y^q+
  \frac{e^{2(\varphi_2-\varphi_1)}}{2}y^\alpha Y^j R_{\beta\alpha jl}^E g^{lq}  \label{equationofparallel2}\\
  (\na^{^{M,E}}_{\pi^*\partial_i}Y)^\alpha &=&  \papa{B^\alpha}{x^i}+B^\beta\Gamma_{i\beta}^{E,\alpha}+c_1y^\alpha Y^jg_{ij}-\frac{1}{2}Y^jy^\beta R^E_{\alpha\beta ij}  \label{equationofparallel3}\\
  (\na^{^{M,E}}_{{\pi^\star}e_\beta}Y)^\alpha &=& \papa{B^\alpha}{y^\beta}+c_2B^\beta y^\alpha + by^\beta B^\alpha+by^\epsilon B^\epsilon\delta_\alpha^\beta  \label{equationofparallel4}
\end{eqnarray}
As the reader may agree, finding a germ of a parallel vector field in general is quite non-trivial even if we require $Y$ to be horizontal or to be vertical.
\begin{prop}
Assume the weight functions $\varphi_1=\varphi_1(r),\ \varphi_2=\varphi_2(r)$ are constant.\\
(i)\ \,The only horizontal parallel vector fields $Y$ on the manifold $E$ are the horizontal lifts of parallel vector fields $Y_0$ of $M$ for which $R^E(Y_0,\ )=0$.\\
(ii)\ \,Likewise, the only vertical parallel vector fields on $E$ are the vertical lifts of parallel sections of $\pi:E\lrr M$.
\end{prop}
\begin{proof}
By hypothesis $a=b=c_1=c_2=0$. First suppose all $B^\alpha=0$ and $\na^{^{M,E}}Y=0$. Then \eqref{equationofparallel3} implies $R^E(Y,\ )=0$. By \eqref{equationofparallel2}, we have $Y^q$ independent of the $y^\beta$, so the vector field is a pull-back: $Y=\pi^*Y_0$. The result that $Y_0$ is $\na^{^M}$-parallel then follows by equation \eqref{equationofparallel1}. Suppose now that all the $Y^j=0$, i.e. $Y$ is vertical. The last equation in the list shows the field $Y$ arises as a vertical pull-back of a section $e$ of $E\lrr M$, i.e. $Y={\pi^\star}e$. The third equation shows that $e$ is ${D}^{^E}$-parallel. Now we conclude from \eqref{equationofparallel1} that in all points of $M$ we must have $R^E(\ ,\ )e=0$. But this is automatic for any parallel section.
\end{proof}
\begin{prop}\label{covaderivativeofxi}
 For all $X\in TE$ we have \,$\na^{^{M,E}}_X\xi=arX^h+(1+br)X^v$.
\end{prop}
\begin{proof}
 Since for all $X\in TE$ we have $A_X\xi=0=\calR(X,\xi)$ and $b=-c_2$, the result follows.
\end{proof}
We may also consider the search for parallel vector fields of the form 
\begin{equation}
Y=f\pi^*Y_0+g{\pi^\star}e+h\xi\ ,
\end{equation}
where $f,g,h$ are functions of $r$ and $Y_0$ and $e$ are sections on $M$, parallel for $\na^{^M}$ and ${D}^{^E}$ respectively. In particular, sections of constant norm.
\begin{teo}\label{conditionsforparallelvectorfields}
Suppose $Y\neq0$.\\
(i)\,\ For $k>1$, the vector field $Y$ is $\na^{^{M,E}}$-parallel if and only if $h=\varphi_1'=f'=0$,\, $fR^E(Y_0,\ )=0$ and also if $\varphi_2'=g'=0$ when $g{\pi^\star}e\neq0$.\\
(ii)\,\ For $k=1$, the vector field $Y$ is $\na^{^{M,E}}$-parallel if and only if

(a) either $\varphi_1'\neq0$ and $f\pi^*Y_0=0$ with $g\|e\|+\sqrt{r}h=0$, or

(b) $\varphi_1'=f'=0$ and $g\|e\|+\sqrt{r}h=ce^{-\varphi_2}$ for some constant $c$.
\end{teo}
\begin{proof}
 We indicate most of the steps. Of course here $g$ denotes a function; not the metric. For functions dependent of $r$, such as $f$, we have $\dx f=2f'\xi^\flat$. Let us denote $\hat{e}={\pi^\star}e$ and simply by $Y_0^h$ the horizontal lift $\pi^*Y_0$. Then, by the hypothesis, the following vanish in all directions $X^h,X^v$:
 \begin{eqnarray*}
 & & \na^{^{M,E}}_{X^h}Y \ =\  ag\xi^\flat(\hat{e})X^h+c_1f\langle X^h,Y_0^h\rangle_{_M}\xi -\frac{1}{2}f\calR(X^h,Y_0^h)+gA_{X^h}\hat{e}+arhX^h
 \end{eqnarray*}
and
\begin{eqnarray*}
 & & \na^{^{M,E}}_{X^v}Y\ =\ 2f'\xi^\flat(X^v)Y_0^h+af\langle\xi,X^v\rangle Y_0^h+fA_{X^v}Y_0^h +2g'\xi^\flat(X^v)\hat{e}+\\
 & &\qquad\qquad \ \ \ \  +gc_2\langle X^v,\hat{e}\rangle\xi+ gb(\xi^\flat(X^v)\hat{e}+\xi^\flat(\hat{e})X^v)+ 2h'\xi^\flat(X^v)\xi+h(1+rb)X^v\ .
\end{eqnarray*}
Now $\langle \na^{^{M,E}}_{X^h}Y,X^h\rangle_{_M}=(ag\xi^\flat(\hat{e})+arh)\|X^h\|^2=0$. For $k>1$ we may take $\xi$, with any norm, orthogonal to $\hat{e}$, so the equation implies $a=0$ or $(h,g)=(0,0)$. If $a\neq0$, then $\varphi_1'\neq0$ and then $\langle \na^{^{M,E}}_{X^h}Y,\xi\rangle_{_E}=c_1f\langle X^h,Y_0^h\rangle_{_M}\|\xi\|^2=0$ implies $fY_0^h=0$ and thus $Y=0$. So we assume $k>1$ and $a=c_1=0$. Then we are left with the vertical and horizontal parts $fR^E(Y_0,\ )=0$ and $gA_{X^h}\hat{e}=0$. The former implies $A_{X^v}Y_0^h=0$, while the latter is equivalent to $gR^E(\ ,\ )\hat{e}=0$, which is satisfied automatically because $e$ is parallel. On the other hand, $\langle \na^{^{M,E}}_{X^v}Y,Y_0^h\rangle=2f'\xi^\flat(X^v)\langle Y_0,Y_0\rangle_{_M}=0$, as it is immediate to see, hence $f'=0$.

Since $k>1$, we then may take $X^v=\hat{e}\perp\xi$, which yields $gc_2\|\hat{e}\|^2\xi+h(1+rb)\hat{e}=0$. These two summands must vanish and if $g\hat{e}\neq0$, then $c_2=-b=0=h$ and $g'=0$. If $g\hat{e}=0$, then because $e$ has constant norm we may assume $g=0$. Now the original equation, resumes again to $h(1+rb)=0$. Integrating $1+r2\varphi_2'=0$ gives $\varphi_2(r)=-\frac{1}{2}\log r+\frac{1}{2}\log c$, with $c>0$ a constant. Then $\varphi_2$ is not bounded at $0$. So we must have $h=0$ instead. And there is nothing left to check.

Now let us see the case $k=1$, where the curvature tensor vanishes: $R^E=0$ by trivial reason. Since $e$ is ${D}^{^E}$-parallel, we may already assume it has norm 1. Then $\xi,\hat{e},X^v$ are collinear, so we may write $X^v=\hat{e}$ and $\xi=\sqrt{r}\hat{e}$. In other words, $\xi^\flat(\hat{e})=\sqrt{r}$. The equation for $\na^{^{M,E}}_{X^h}Y$ yields
$ag\sqrt{r}X^h+c_1f\langle X^h,Y_0^h\rangle\xi+arhX^h=0$. Hence $c_1fY_0=0$ and $a(g\sqrt{r}+rh)=0$. The equation for $\na^{^{M,E}}_{X^v}Y$ gives $2f'\sqrt{r}Y_0+af\sqrt{r}Y_0=0$ and
\[ 2g'\sqrt{r}+gc_2\sqrt{r}+2gb\sqrt{r}+2h'r+h+rbh=0\ .\]
Since $b=-c_2=2\varphi_2'$, this is equivalent to
\[  2\varphi_2'(g\sqrt{r}+rh)+2g'\sqrt{r}+2h'r+h=0 \ .\qquad(*)\]
Of course this has to simplify further. Letting $\psi=g+\sqrt{r}h$, then ($*$) is equivalent to $\varphi_2'\psi+\psi'=0$.
If $c_1\neq0$, then $fY_0=0$ and $a\neq0$ and so $f'=0$ and $g\sqrt{r}+rh=0$, this is $\psi=0$. Finally, if $c_1=0$ then $a=0$ and the equations yield $f'=0$ and $\psi=ce^{-\varphi_2}$ for some constant $c$.
\end{proof}

Regarding the more general equation for a Killing vector field, i.e. a vector field $X$ such that the tensor field $\calL_X{\mg_{_{M,E}}}$ vanishes identically, equivalently, such that
\begin{equation}
 \mg_{_{M,E}}(\na^{^{M,E}}_YX,Z)+\mg_{_{M,E}}(Y,\na^{^{M,E}}_ZX)=0,\ \forall Y,Z\in\XIS(E)\ ,
\end{equation}
we cannot go much farther. We find such Lie derivative to be equal to (the meaning of $\calL_{X^h}{\pi^*g_{_M}}$ and $\calL_{X^v}{\pi^*g_{_E}}$ being analogous)
\begin{equation}
\begin{split}
  e^{2\varphi_1}(\calL_{X^h}{\pi^*g_{_M}})(Y,Z)+ e^{2\varphi_2}(\calL_{X^v}{{\pi^\star}g_{_E}})(Y,Z)
 +2a e^{2\varphi_1}\xi^\flat(X)\langle Y,Z\rangle_{_M}+\qquad\qquad\\
 +2b e^{2\varphi_2}\xi^\flat(X)\langle Y,Z\rangle_{_E} + e^{2\varphi_2}\langle\calR(X,Z),Y\rangle_{_E}+e^{2\varphi_2}\langle\calR(X,Y),Z\rangle_{_E}
\end{split}\ .
\end{equation}

Although infinitesimal isometries of the space $E$ are indeed difficult to describe, we have the following quite immediate construction. Suppose we have another Riemannian manifold $M_1$ together with a vector bundle $E_1\lrr M_1$ endowed with a metric structure and metric connection $D^{^{E_1}}$. Suppose also we have a parallel vector bundle isometry $\hat{f}$ along an isometry $f$ of the base manifolds:
\begin{equation}
 \begin{array}{rcl}
  E &\stackrel{\hat{f}}{\lrr}&E_1 \\
  \pi\downarrow & &\downarrow\pi_1\\
  M&\stackrel{f}{\lrr} & M_1
 \end{array}\ \ .
\end{equation}
We recall parallel morphism means that $f^*D^{^{E_1}}\circ\hat{f}=\hat{f}\circ D^{^E}$.
\begin{teo}
 In the above conditions, for the given same pair of functions $\varphi_1,\varphi_2$ on the radius of $E$ and $E_1$, the map $\hat{f}:(E,\mg_{_{M,E}})\lrr(E_1,\mg_{_{M_1,E_1}})$ is an isometry.
\end{teo}
\begin{proof}
 Using the connection $f^*D^{^{E_1}}$ and a general procedure, one may reduce the problem to a parallel isometry $\hat{f}:(E,D^{^E})\lrr (E_1,D^{^{E_1}})$ of vector bundles over $M$ and along the identity of $M$. In this setting, this map may be raised to a vector bundle isometry $\hat{f}^{\mathrm{up}}(e,e')=(\hat{f}(e),\hat{f}(e'))$ of ${\pi^\star}E$ over $E$ onto the respective $\pi_1^\star E_1$ over $E_1$. Then we have $\hat{f}^*\xi_1=\hat{f}^{\mathrm{up}}\circ\xi$ and by hypothesis
 ${{\hat{f}}^*\pi_1^* D^{^{E_1}}}_X\circ\hat{f}^{\mathrm{up}}= \hat{f}^{\mathrm{up}}\circ{\pi^\star}D^{^E}_X,\ \forall X\in TE$. Finally it follows that 
 \begin{eqnarray*}  
 {\pi_1^* D^{^{E_1}}}_{\hat{f}_*X}\xi_1 &=& 
 {{\hat{f}}^*\pi_1^*D^{^{E_1}}}_X\hat{f}^*\xi_1 \\
 &=&  {{\hat{f}}^*\pi_1^*D^{^{E_1}}}_X\hat{f}^{\mathrm{up}}\xi
         \ =\ {\hat{f}}^{\mathrm{up}}({\pi^\star}D^{^E}_X\xi)\ .
 \end{eqnarray*}
 Hence the respective $D^{^E},D^{^{E_1}}$ horizontal subspaces are $(\mg_{_{M,E}},\mg_{_{M,E_1}})$-isometrically preserved by $\dx{\hat{f}}={\hat{f}}_*$ since this derivative is \textit{essentially} the identity on horizontals and since $r=\|e\|_{_E}^2=\|\hat{f}(e)\|_{_{E_1}}^2$ at all points $e\in E$ and the weight function $\varphi_1$ is the same. Regarding the vertical directions, the differential $\dx{\hat{f}}$ is that of a linear map, precisely $\hat{f}^{\mathrm{up}}$, and $\hat{f}$ is an isometry on the fibres by previous similar reason now with the weight function $\varphi_2$. Thus described the whole picture, we conclude the given bundle morphism is a manifold isometry.
\end{proof}
It is quite often the case that one has an isometry $f$ of $M$ and that $E\subset T^{p,q}M$,  $p,q\in\N$, is a sub-vector bundle of the tangent $(p,q)$-tensors on $M$, such that $f_*(E_x)=E_{f(x)},\ \forall x\in M$.
\begin{coro}\label{Isometriesarenaturallylifted}
 For any two functions $\varphi_1,\varphi_2$ of the squared-radius $r$, we have a 1-1 map
\begin{equation}
  \mathrm{Isom}(M,g_{_M})\hookrightarrow \mathrm{Isom}(E,\mg_{_{M,E}})\ .
\end{equation}
\end{coro}
\begin{proof}
The uniqueness of the Levi-Civita connection of $M$ implies it is an invariant connection for any given isometry $f$ of the base (this fundamental property, often taken for granted, contrasts for instance with symplectic geometry and symplectic connection theory). Then we are in the conditions above with the map $\hat{f}:E\lrr E$ induced by the differential $f_*:E\lrr f^*E$. (The term invariant connection, also used below in proposition \ref{totallygeodesicsections}, comes from \cite[vol. I, \S5]{KobNomi}).
\end{proof}

\subsection{Totally geodesic submanifolds}

We continue to deduce some properties of the metric.
\begin{prop}\label{totallygeodesicfibresandparallelsubbundles}
The Riemannian metric $\mg_{_{M,E}}$ and its Levi-Civita connection $\na^{^{M,E}}$ satisfy the following properties:\\
(i)\,\ The zero section  $M\equiv O_M\subset E$ is totally geodesic.\\
(ii)\,\ The fibres of $E$ are totally geodesic.\\
(iii)\,\ The following three assertions are equivalent: the vertical distribution $\calV\subset TE$ is $\na^{^{M,E}}$-parallel; the horizontal distribution $\calH^{{D}^{^E}}$ is $\na^{^{M,E}}$-parallel; $\varphi_1$ is a constant and ${D}^{^E}$ is flat.
\end{prop}
\begin{proof}
(i) Trivial.\\
(ii) Equation \eqref{equationofparallel2} vanishes immediately if there is no horizontal part of $Y$. This condition of $\na^{^{M,E}}_{Z}Y\in\Gamma(\calV)$ for $Z,Y\in\Gamma(\calV)$ is also trivial to see from \eqref{theLeviCivitaconnection}, if one prefers. We conclude $\na^{^{M,E}}$ is also the Levi-Civita connection of the fibres, with metric $e^{2\varphi_2}{\pi^\star}g_{_E}$, so they are totally geodesic.\\
(iii) We prove that the two first assertions are both equivalent to the third. First, because of (ii), we are left to see $\na^{^{M,E}}_XY$ is vertical for any $X$ horizontal and $Y$ vertical. Taking $Z$ horizontal and looking at $e^{2\varphi_1}\langle \na^{^{M,E}}_XY,Z\rangle$, we get a symmetric and a skew-symmetric part:
\[ e^{2\varphi_1}\langle \na^{^{M,E}}_{X^h}Y,Z^h\rangle=e^{2\varphi_1}a\xi^\flat(Y)\langle X^h,Z^h\rangle+\frac{e^{2\varphi_2}}{2}\langle\calR(X^h,Z^h),Y\rangle \]
which must vanish in\-de\-pen\-den\-tly, giving the conditions since $a=2\varphi_1'$. They can also be read from \eqref{equationofparallel1}. Assuming now the second assertion and taking $Y$ horizontal in (\ref{equationofparallel3},\ref{equationofparallel4}), thus with all $B_\alpha=0$, the equations tell us again we must have $c_1=0$ and $R^E=0$ if $\na^{^{M,E}}Y$ is going to be horizontal, and reciprocally.
\end{proof}
The integrability of the horizontal distribution is independent of the metric.
The zero section $O_M$ is an important totally geodesic submanifold of $E$, which shall deserve further attention in the next chapter. It may be understood as an example of a \textit{soul} in the famous article of Cheeger and Gromoll, cf. \cite{Tapp}. From proposition \ref{covaderivativeofxi} we see that $\na^{^{M,E}}\xi=0$ on $E\backslash O_M$ if and only if $\varphi_1'=0$ and there $\exists c>0$ a constant such that
\begin{equation}
 \varphi_2(r)=-\frac{1}{2}\log r\,+ \frac{1}{2}\log c\ .
\end{equation}
However, we note $e^{2\varphi_2}=\frac{c}{r}$ is not bounded at the 0 section of $E$, failing our wishes to have smooth well-behaved metrics. Hence $\xi$ is never parallel because of $O_M$. We have the following result for later applications. It is the converse question from theorem \ref{conditionsforparallelvectorfields}.
\begin{teo}\label{parallelvectorfieldsonEandM}
 If the manifold $E$ admits a $\na^{^{M,E}}$-parallel non-vertical vector field $Y$, then $M$ admits a $\na^{^M}$-parallel vector field. In other words, every $\mg_{_{M,E}}$-parallel vector field over $E$ restricts over $O_M$ to an orthogonal sum of a parallel vector field of $M$ and a parallel section of $E$.
\end{teo}
\begin{proof}
 One just reads through \eqref{equationofparallel1} and \eqref{equationofparallel3}, consistent with (i) from proposition \ref{totallygeodesicfibresandparallelsubbundles}.
\end{proof}
 
A problem of the same kind is to try to find other sections $e:M\rr E$ which embed $M$ as a totally geodesic submanifold $e(M)=M^e$ of the Riemannian manifold $(E,\mg_{_{M,E}})$. It is easy to deduce the canonical decomposition (notice $e^*{\pi^\star}{D}^{^E}e^*\xi={D}^{^E}e$)
\begin{equation}
 \dx e(X)=X^h+{D}^{^E}_Xe\ \in e^*TE\:,\ \ \forall X\in T_xM\ .
\end{equation}
Hence we have a description of $T_{e_x}M^e\subset T_{e_x}E$ for each $x\in M$. We end this section with a particular case (which might induce further relevant study). It is interesting since it brings into the equations a new general Hessian operator $H^{^E}$.
\begin{prop}\label{totallygeodesicsections}
(i) Let $\varphi_1,\varphi_2$ be constants and suppose that $R^Ee=0$. Then $M^e$ is a totally geodesic submanifold of $E$ if and only if
\begin{equation}\label{halfcurvatureequation}
H^{^E}(X,Y)e={D}^{^E}_X{D}^{^E}_Ye-{D}^{^E}_{\na^{^M}_XY}e=0 \ ,\qquad\forall X,Y\in TM\ . 
\end{equation}
(ii)  Let $\varphi_1,\varphi_2$ be constants and suppose that $e_0$ is a $D^{^E}$-parallel section. Then the translation map $t:E\lrr E$, $t(e)=e+e_0$, is an invariant map of $\na^{^{M,E}}$.
\end{prop}
\begin{proof}
(i) Letting $X^h+{D}^{^E}_Xe$ and $Y^h+{D}^{^E}_Ye$ be any two vectors tangent to $M^e$, then the condition is that the following vector must also be tangent to the same submanifold: 
 \[ \na^{^{M,E}}_{X^h+{D}^{^E}_Xe}(Y^h+{D}^{^E}_Ye)=\na^{^M}_XY+{D}^{^E}_X{D}^{^E}_Ye\ .\]
The right hand side follows from computations; for instance one must verify that $A_{X^h}{D}^{^E}_Ye=0$. The formula yields the result immediately.\\
(ii) First, recalling the hypothesis, we conclude with a little computation that $t_*$ preserves $\calH^{D^{^E}}$. It follows that essentially $t_*$ is the \textit{identity} map on $TE$, which implies easily that $t\cdot\calR=\calR$. Since the tensor $C=0$ and the weights are constant, we also have $A$ and hence $\na^{^{M,E}}$ invariant by $t$.
\end{proof}
Clearly, the left hand side of \eqref{halfcurvatureequation} is the half of $R^E_{X,Y}e=H_{X,Y}e-H_{Y,X}e$. Also, taking the trace, we may say $e$ is harmonic. We note this generalized Hessian and its symmetric part are tensorial in $X,Y$. These operators could be better studied in the theory of connections on vector bundles.

\subsection{The geodesics}
\label{Thegeodesics}

We recover the trivialization of $E$ introduced earlier. It is given by a neighbourhood $\inv{\pi}(U)\simeq U\times\R^k$ where $U$ is the domain of a chart $x$ of $M$. We also use an orthonormal frame $\{e_\alpha\}_{\alpha=1,\ldots,k}$, formed by sections of $E$ on $U$, as introduced in the same section \ref{somepropertiessection}. 

A curve $\gamma=\gamma(t),\ t\in\R$, with image in $\inv{\pi}(U)\subset E$ may then be written in the coordinates, functions of $t$, as a map:
\begin{equation}
 \gamma=(\gamma^1,\ldots,\gamma^m,y^1,\ldots,y^k)\ .
\end{equation}
In  the following we use summation convention for repeated indices and denote by $\gammadot^i$ the derivative with respect to $t$. The tautological vector field $\xi$ along $\gamma$ reads $\xi_{\gamma}=y^\alpha{\pi^\star}e_\alpha$. When $\gamma$ defines a section $y=y^\alpha e_\alpha$ of $E\lrr M$ along $\pi\circ\gamma=(\gamma^1,\ldots,\gamma^m)$, then along this same curve we have
\begin{equation}
{D}^{^E}_{\partial_t}y=\ydot^\beta e_\beta+\gammadot^iy^\alpha\Gamma_{i\alpha}^{E,\beta} e_\beta=z^\beta e_\beta\ . 
\end{equation}
It is indeed useful to define 
\begin{equation}
 z^\beta=\ydot^\beta+\gammadot^iy^\alpha\Gamma_{i\alpha}^{E,\beta} \ .
\end{equation}
\begin{teo}\label{eqgeodesics}
 The curve $\gamma$ is a geodesic of $\mg_{_{M,E}}$ if and only if, $\forall{1\leq p\leq m},\ {1\leq\alpha\leq k}$, we have
 \begin{equation}\label{eqgeodesics2}
  \begin{cases}
  \gammadotdot^p+\gammadot^i\gammadot^j\Gamma_{ij}^{M,p} +2a\gammadot^p z^\beta y^\beta+ e^{2\varphi_2-2\varphi_1}\gammadot^iz^\beta y^\mu R^E_{\beta\mu iq}g^{qp}=0\\
  \dot{z}^\alpha+\gammadot^i\gammadot^jc_1g_{ij}y^\alpha+\gammadot^i z^\beta\Gamma_{i\beta}^{E,\alpha}-bz^\beta z^\beta y^\alpha 
  +2bz^\alpha z^\beta y^\beta=0
  \end{cases}\ .
 \end{equation}
\end{teo}
\begin{proof}
A geodesic of $\mg_{_{M,E}}$ is a curve which satisfies ${\gamma^*\na^{^{M,E}}}_{\partial_t}\gammadot=0$ (introducing $\gamma^*$ is the same as saying \textit{along} $\gamma$), so first we deduce from \eqref{decompositionofchartvectorfields} the canonical decomposition 
 \begin{align*}
  \gammadot&=\gammadot^i\partial_i+\ydot^\beta\partial_{y^\beta}\\
  &=\gammadot^i(\pi^*\partial_i+y^\alpha\Gamma_{i\alpha}^{E,\beta}{\pi^\star}e_\beta) +\ydot^\beta{\pi^\star}e_\beta\\
  &=\gammadot^i\pi^*\partial_i+z^\beta{\pi^\star}e_\beta 
 \end{align*}
(notice that this is essentially $\gammadot=\gammadot^i\pi^*\partial_i+\pi^\star(D^{^E}_{\partial_t}y)$).
Then 
\begin{equation*}
 {\gamma^*\na^{^{M,E}}}_{\partial_t}\gammadot=\gammadotdot^i\pi^*\partial_i+ \gammadot^i\na^{^{M,E}}_{\gammadot}\pi^*\partial_i+\dot{z}^\beta{\pi^\star}e_\beta +z^\beta \na^{^{M,E}}_\gammadot{\pi^\star}e_\beta\ ,
\end{equation*}
and since we have 
\begin{align*}
\na^{^{M,E}}_{{\pi^\star}e_\beta}\pi^*\partial_i=ay^\beta\pi^*\partial_i+A_{{\pi^\star}e_\beta}\pi^*\partial_i=ay^\beta\pi^*\partial_i+\frac{e^{2\varphi_2-2\varphi_1}}{2}y^\mu R^E_{\beta\mu ij}g^{jq}\pi^*\partial_q\ ,
\end{align*}
we deduce the two summands
\begin{align*}
\gammadot^i\na^{^{M,E}}_\gammadot\pi^*\partial_i&=\gammadot^i\gammadot^j\bigl(\Gamma_{ji}^{M,l}\pi^*\partial_l+c_1g_{ij}y^\alpha{\pi^\star}e_\alpha-\frac{1}{2}\calR(\pi^*\partial_j,\pi^*\partial_i)\bigr)+\gammadot^iz^\beta \na^{^{M,E}}_{{\pi^\star}e_\beta}\pi^*\partial_i\\
&= \gammadot^i\gammadot^j\Gamma_{ij}^{M,l}\pi^*\partial_l
+\gammadot^i\gammadot^jc_1g_{ij}y^\mu{\pi^\star}e_\mu 
+\gammadot^iz^\beta (ay^\beta\pi^*\partial_i 
+\frac{e^{2\varphi_2-2\varphi_1}}{2}y^\mu R^E_{\beta\mu ij}g^{jq}\pi^*\partial_q)
 \end{align*}
and
\begin{align*}
 z^\beta \na^{^{M,E}}_\gammadot{\pi^\star}e_\beta&=z^\beta\gammadot^j(\Gamma_{j\beta}^{E,\mu}{\pi^\star}e_\mu +ay^\beta\pi^*\partial_j+\frac{e^{2\varphi_2-2\varphi_1}}{2}y^\mu R^E_{\beta\mu jl} g^{lq}\pi^*\partial_q)+\\
  &\qquad +z^\beta z^\nu(c_2\delta_\nu^\beta y^\tau{\pi^\star}e_\tau+
 by^\nu{\pi^\star}e_\beta+by^\beta{\pi^\star}e_\nu)\ .
\end{align*}
 Recalling $c_2=-b$, summing and contracting, finishes the proof.
\end{proof}
We recall that $\Gamma^M,\Gamma^E$ and $R^E$ depend only of the $\gamma^i$. Also the geodesics of $M$ become geodesics of $O_M$, the zero section, as expected by earlier findings. Other lifts are quite `singular'.
\begin{prop}
Let $\gamma$ be a curve in $E$  which defines a \textit{parallel section} $y$ along the curve $\tau=\pi\circ\gamma$, thus having $\|y\|^2_{_E}=r_0$ a constant. Moreover, we assume that $r_0\neq0$. Then $\gamma$ is a geodesic of $\mg_{_{M,E}}$ if and only if $\tau$ is a geodesic of $M$ and  ${\varphi'_1}(r_0)=0$.
\end{prop}
\begin{proof}
 This is immediate from \eqref{eqgeodesics2}, since the assumption is $z^\alpha=0,\ \forall1\leq\alpha\leq k$, and the term $\gammadot^i\gammadot^jc_1g_{ij}=\|\tau\|_{_M}^2c_1$ (notice some $y^\alpha\neq0$) varies only with $c_1(r)$ for any geodesic $\tau$.
\end{proof}
It is interesting to notice the case $a=c_1=0$, i.e., the case $\varphi_1'(r)=0, \forall r$. Since parallel sections of $E$ exist along any curve in $M$, we have lifts of geodesics of $M$ to geodesics of $E$ with any given initial 1st order conditions. In the next subsection we look at the case when all $\gammadot^i=0$. These are the vertical geodesics, described below in formula \eqref{verticalgeodesic}.

Also we observe that while the first equation in \eqref{eqgeodesics2} is similar to that of a Jacobi vector field $y$ along a curve, the second corresponds with ${D}^{^E}{D}^{^E}y=0$.

Regarding the completeness of the metric $\mg_{_{M,E}}$, we have some observations on a statement which still aims for a rigorous proof. Recall the hypothesis that $\varphi_i,\ i=1,2$ are smooth at $r=0$ on the right. We then conjecture that any such $\mg_{_{M,E}}$ is complete if and only if the metric $g_{_M}$ on $M$ is complete and also the totally geodesic fibres are complete. Our argument relies on the results that Riemannian completeness is a question of the induced metric space, that solutions for the above system do exist and that the topology of $U\times\R^k$ does not prohibit their continuous development through any given instant. The completeness by Cauchy sequences on the base and the bundle transition functions assure the smooth development up to infinity of geodesics contained in $E$.

In particular we believe the previous assertion remains true regardless of ${D}^{^E}$ being also complete or not (we have in mind the notion of a \textit{complete} connection on a vector bundle as a connection for which parallel transport in $E$ is continuously defined along any given curve in $M$)\footnote{These interactions are important, specially for the above problem when we think of the pseudo-Riemannian case. However, even for this situation, for the weighted Sasaki pseudo-Riemannian structures, defined by obvious sign change in \eqref{theweightedSasakimetric}, we believe the metric completeness of the base manifold is still the sufficient condition, with the same arguments as above.}.

There seems to be no reference for this problem, namely among the many studies of all the generalized Sasaki metrics found in the literature. For geodesics of tangent sphere bundles with classical Sasaki metric ($\varphi_1=\varphi_2=0$), one may see \cite{BBNV} and try to bring those results into the present setting. The completeness stands again conjecturally when we have a complete but non-compact base.

\subsection{Spherically symmetric metrics on $\R^k$}

A vertical geodesic of $\mg_{_{M,E}}$ is a geodesic which lies in the fibres of $E$. In virtue of proposition \ref{totallygeodesicfibresandparallelsubbundles}, any vertical geodesic is equivalent to a geodesic of $E$ which is tangent to the fibres of $E$ at just one point. We can analyse these curves in the vector-space manifold $\R^k$ with metric $e^{2\varphi_2(r)}((\dx y^1)^2+\cdots+(\dx y^k)^2)$. The metric, say $g_{_{\varphi_2}}$, clearly has spherical symmetry (or rotational symmetry, both terms appear in the literature with the same meaning). Since we could not find its basic computations elsewhere, we write some results here for completion. From theorem \ref{eqgeodesics} we immediately write the geodesic equations:
\begin{equation}\label{verticalgeodesic}
 \ydotdot^\alpha+2b\ydot^\alpha\ydot^\beta y^\beta-b\ydot^\beta\ydot^\beta y^\alpha=0\ ,\quad\forall1\leq\alpha\leq k\ .
\end{equation}
We remark ${\dot{\varphi}_2}=2\ydot^\beta y^\beta\varphi_2'=b\ydot^\beta y^\beta$ and $r=y^\alpha y^\alpha$ (cf. \eqref{derivativeofr}). As it can be seen, resuming with our main study, it is not easy to find the vertical geodesics of the metric $\mg_{_{M,E}}$. Of course the case $b=0$ is well-known.

For completion of exposition we give the sectional, Ricci and scalar curvatures of the metric $g_{_{\varphi_2}}$. With the free coordinates $y^\alpha$ of $\R^k$, we deduce from \eqref{theLeviCivitaconnection} that $\na^{^{M,E}}_\beta\partial_\nu = -b\delta_\beta^\nu y^\mu\partial_\mu+by^\beta\partial_\nu+  by^\nu\partial_\beta$ and hence that (no summation on repeated indices here)
\begin{equation}
 \begin{split}
  R^{\na^{^{M,E}}}_{\alpha\beta\alpha\beta}&=2\papa{b}{y^\alpha}y^\beta\delta_\alpha^\beta- \papa{b}{y^\alpha}y^\alpha-\papa{b}{y^\beta}y^\beta
 +b(1+br)\delta_\alpha^\beta-b(1+br)+\\&\qquad\qquad+ b\delta_\alpha^\beta-b+
b^2(y^\beta y^\beta-2y^\alpha y^\beta\delta_\alpha^\beta+y^\alpha y^\alpha)\ .
 \end{split}
\end{equation}
For any plane $\Pi\subset T_y\R^k$ spanned by $\partial_\alpha,\partial_\beta$ with $\alpha\neq\beta$, we find (no summation convention here) the sectional curvature:
\begin{equation}
\begin{split} \label{curvaturaseccionalvvv}
 k^{g_{_{\varphi_2}}}(\Pi)&=e^{-2\varphi_2}R^{\na^{^{M,E}}}_{\alpha\beta\alpha\beta}\\
 &=e^{-2\varphi_2}\biggl(-\papa{b}{y^\alpha}y^\alpha-\papa{b}{y^\beta}y^\beta-2b-b^2r+b^2(y^\beta y^\beta+y^\alpha y^\alpha)\biggr)\\
 &=e^{-2\varphi_2}\bigl((b^2-2b')(y^\alpha y^\alpha+y^\beta y^\beta)-2b-b^2r\bigr)\\
 &=4e^{-2\varphi_2}\bigl(({\varphi_2'}^2-\varphi_2'')(y^\alpha y^\alpha+y^\beta y^\beta)-\varphi_2'-r{\varphi_2'}^2\bigr) \ .
\end{split} 
\end{equation}
For $k=2$ the metric is flat if and only if $\varphi_2''r+\varphi_2'=0$. Non-trivial solutions are ill-defined metrics: $\varphi_2=-l\log r\,+L$ with constant $l,L$.
 The Ricci curvature is given by
\begin{equation}
\begin{split}
 \ric^{g_{_{\varphi_2}}}(\partial_\beta,\partial_\beta) &=\sum_{\alpha\neq\beta}R^{\na^{^{M,E}}}_{\alpha\beta\alpha\beta}
 =(b^2-2b')\bigl(r+(k-2)(y^\beta)^2\bigr)-(k-1)(b^2r+2b)\\
 &\quad\quad=4({\varphi_2'}^2-{\varphi_2''})\bigl(r+(k-2)(y^\beta)^2\bigr)-4(k-1)({\varphi_2'}^2r+{\varphi_2'})
\end{split}
\end{equation}
and (easy also to deduce from \cite[Theorem 1.159, formula f]{Besse})
\begin{equation}
\begin{split}
 \Scal^{g_{_{\varphi_2}}}&=e^{-2\varphi_2}(k-1)\bigl(2b^2r -4b'r-kb^2r-2kb\bigr)\\
 &=-4e^{-2\varphi_2}(k-1)\bigl(r{\varphi_2'}^2(k-2)+\varphi_2'k+2r\varphi_2''\bigr)\ .
\end{split}
\end{equation}
We can easily guarantee conditions in order to have $g_{_{\varphi_2}}$ globally with negative scalar curvature.

\subsection{Cheeger-Gromoll, Musso-Tricerri and generalized Bergery metric}
\label{CGMTagBm}

In the famous paper \cite{MusTri}, E.~Musso and F.~Tricerri introduced a metric on the total space of $TM\lrr M$ which then they attribute to \cite{CheeGromoll}. The well-known name of Cheeger-Gromoll for a metric on the tangent bundle of any given Riemannian manifold is quite surprising\footnote{Indeed, after reading both articles, the present author does not find the significant reason for this attribution and he seems not to be the only; in \cite{KaziSali} the choice is referred as a matter of inspiration.}.

The Cheeger-Gromoll metric is nowadays an extensively studied and generalized object, cf. \cite{Abb1,AbbCalva,AbbSarih,Alb2008,BenLoubWood1,BenLouWood2,KaziSali,KozNied,Munteanu,Seki,TahVanWat} just to reference a few works. The source for this wealth of studies is the existence of a 1-form $\xi^\flat$ (recall this ${}^\flat$-duality throughout the text is relative to the un-weighted metric $\pi^*g_{_M}\oplus{\pi^\star}g_{_{E}}$). We remark $\xi^\flat$ agrees with the metric-dual to the Liouville 1-form $\lambda$ on the manifold $E^*=T^*M$, which induces the well-known symplectic structure $\dx\lambda$ over the most important phase-space of Hamiltonian mechanics (relations within this point and the geodesic spray vector field may be seen in \cite{Alb2012}).

Notice we are now talking about $E=TM$ and so another natural 1-form $\theta$ may be brought into the picture; $\theta$ is precisely the horizontal dual form of the form $\xi^\flat$. One of the most general Cheeger-Gromoll metrics appearing in references above is thus defined by
\begin{equation}
 G=\mg_{_{M,TM}}+f_3\xi^\flat\otimes\xi^\flat+f_4\theta\otimes\theta+ e^{2\varphi_5}(\theta\otimes\xi^\flat+\xi^\flat\otimes\theta)
\end{equation}
with $f_3,f_4,\varphi_5$ scalar functions such that $e^{2\varphi_2}+f_3>0,\ e^{2\varphi_1}+f_4>0$.

Now, we are interested in the total space of any vector bundle $E\lrr M$ given in general. Hence there is no reason for considering the form $\theta$. We shall give the name \textit{Musso-Tricerri metric} on the manifold $E$ to the metric defined by
\begin{equation}\label{metricMussoTricerri}
 \mg_{_{M,E}}^{\mathrm{MT}}=\mg_{_{M,E}}+f_3\xi^\flat\otimes\xi^\flat
\end{equation}
with $e^{2\varphi_2}+f_3>0$ and $f_3:E\lrr\R$ smooth. The Levi-Civita connection is of the form $\na^{^{M,E}}+L$, with $L$ non trivial but not difficult to find. We shall not carry this study here, which stems from the tangent bundle particular setting.

We recall now that L.~B\'erard Bergery introduced a Riemannian structure which is defined as a particular case of what follows next (cf. \cite{BBergery} and \cite{BeleWei} or \cite{Tapp} for recent applications). First, Euclidean space $\R^k$ is identified with $S^{k-1}\times\R^+_0/\sim$, with the sphere $S^{k-1}\times\{0\}$ collapsed to a point. Euclidean metric at any point $(x,t)$ in polar coordinates is easily seen to be $g_{_{\mathrm{Euc}}}=t^2g_{_{S^{k-1}}}+(\dx t)^2$. A radial deformation of this metric is then achieved by
\begin{equation}
 g_{_f}=f^2(t)g_{_{S^{k-1}}}+(\dx t)^2
\end{equation}
for any given smooth odd function $f=f(t):\R\lrr\R$.

Now given the vector bundle $\pi:E\lrr M$ and any two smooth functions $\varphi_1(r),\varphi_2(r)$, satisfying the usual hypothesis, and supposing we are given the same function $f$ as above, we define a \textit{generalized Bergery metric} on $E$ through the canonical splitting \eqref{canonicaldecompo} and the formula
\begin{equation}\label{metricBergery}
  \mg_{_{M,E,f}}=e^{2\varphi_1}\pi^*g_{_M}\oplus e^{2\varphi_2}{\pi^\star}{g_{_{E}}}_f\ .
\end{equation}
This is clearly the usual metric $\mg_{_{M,E}}$ with radial weights we have been treating, when $f$ is the identity. Recall we have denoted the squared-norm as $r=t^2$ in previous sections.

The above generalizes the Riemannian metric of B.~Bergery, which is constructed only, to the best of our knowledge, on the trivial flat vector bundle $E=M\times\R^k$ and the function $\varphi_2=0$.

The condition to have a complete metric in the trivial bundle case is the following: $M$ must be complete, as well as the fibres, and also $\varphi_1$ smoothly extendible to 0, cf. \cite{BBergery}. For the fibres, completeness is the same as $f(t)>0,\ \forall t>0$, and $f'(0)=1$, cf. \cite{BBergery}.
\begin{teo}\label{metricsthesame}
 Any \emph{complete} generalized Bergery metric on $E$ in the above conditions is conformally equivalent to the Musso-Tricerri metric. More precisely, with $t=\sqrt{r}$ and $f$ function of $t$, then  $f^2/t^2$ is smooth on $\R$ and
 \begin{equation}\label{conformalequivalence}
  \mg_{_{M,E,f}}=\frac{f^2}{r} \Bigl(\frac{r}{f^2} e^{2\varphi_1}\pi^*g_{_M}+e^{2\varphi_2}{\pi^\star}g_{_E}+ e^{2\varphi_2}\bigl(\frac{1}{f^2}-\frac{1}{r}\bigr)\xi^\flat\otimes\xi^\flat\Bigr)\ .
 \end{equation}
 \end{teo}
\begin{proof}
The problem lies first within the fibres so we simplify computations by hiding the horizontal part. Recall $\dx r=2\xi^\flat$ from \eqref{derivativeofr}. Hence
   \begin{align*}
{\pi^\star}g_{_{E,f}}&=f^2(\sqrt{r})g_{_{S^{k-1}}}+(\dx \sqrt{r})^2\\
&=f^2g_{_{S^{k-1}}}+\frac{1}{r}\xi^\flat\otimes\xi^\flat\\
&=\frac{f^2}{r}\bigl(rg_{_{S^{k-1}}}+\frac{1}{r}\xi^\flat\otimes\xi^\flat\bigr)
+\bigl(1-\frac{f^2}{r}\bigr)\frac{1}{r}\xi^\flat\otimes\xi^\flat\\
&=\frac{f^2}{r}\Bigl({\pi^\star}g_{_{E}}+ 
\bigl(\frac{1}{f^2}-\frac{1}{r}\bigr)\xi^\flat\otimes\xi^\flat\Bigr)\ .
\end{align*}
Hence the smooth conformal factor expressed in the formula, at least away from $t=0$. The Mac-Laurin expansion $f(t)=t+\frac{f'''(0)}{6}t^3+\ldots$ shows that $\psi(t)=f^2(t)/t^2$ is smooth everywhere. In case $f$ is not analytic, the formula still indicates that $\psi$ is continuously differentiable at 0 --- which can be proved using the Cauchy rule. The conformal factor is thus smooth with respect to $t$, as required, and well-defined. Finally the summand 
$(\frac{1}{f^2}-\frac{1}{r})\xi^\flat\otimes\xi^\flat$, required for the Musso-Tricerri metric, is smooth by the same reason and has norm $\frac{r}{f^2}-1$, which tends to 0 when $t\rr0$. Regarding the weights, the smoothness of the metric is now proved.
\end{proof}

\section{The Riemannian curvature of $\mg_{_{M,E}}$}

\subsection{Two observations}
\label{Twoobserv}

Let us begin this section with a model example and a remark which is somewhat related to curvature.

Suppose $M$ is a K\"ahler manifold of real dimension $m$, with parallel K\"ahler form $\omega\in\Gamma(M;E)$ where $E=\Lambda^2T^*M$. The rank $k=({}^m_{2})$ vector bundle $E$ is endowed with an induced metric structure on its fibres and with a compatible metric connection. We may hence consider the metric $\mg_{_{M,E}}$ with constant coefficients or any other. For the moment we let $\varphi_1,\varphi_2$ be constant. The image $\omega(M)=M^\omega$ inside $E$ is a totally geodesic submanifold (cf. proposition \ref{totallygeodesicsections}). Its tangent space is $\calH^{{D}^{^E}}$. Thus a trivialization chart of $E$ corresponds with a de \!Rham decomposition of $E$ near each point $t\omega_x,\ \forall x\in M,\  t\in\R$. We consider now the vector bundle $E^0=\omega^\perp$ over $M$, i.e. $E^0_x=\{e\in E:\ \langle e,\omega\rangle_{_E}=0\}$. Then $E^0$ is a submanifold of dimension $m+k-1$. Since it is a parallel sub-vector bundle, the manifold $(E^0,\mg_{_{M,E^0}})$ is totally geodesic in $(E,\mg_{_{M,E}})$. Its tangent bundle is given by the perpendicular to the vector field  $Y={\pi^\star}\omega$, the vertical lift, which is a $\mg_{_{M,E}}$-parallel vector field on $E$, as follows from theorem \ref{conditionsforparallelvectorfields}.

The above picture is quite irrelevant, was it not true that one can find it with any vector bundle with a parallel section, such as the obvious compatible almost complex structure $J\in\Gamma(\sol(TM))$ or even the metric $g_{_M}\in\Gamma S^2(T^*M)$ in the general real manifold context. We remark the holonomy of each respective $E^0$ is closely related to that of $E$. Moreover, further relations between algebra and geometry follow from the natural Riemannian bundle structures with canonical metric structure, such as $\Lambda^pE,\ S^pE$ or the tensor product of two distinct Riemannian vector bundles.

The 5-dimensional manifold $S^2T^*M$ or its 4-dimensional subspace $E^0=\{g_{_M}\}^\perp$, associated to any given Riemann surface, have interesting computable geometry which may be studied elsewhere. Hopefully one might be able to find $\mathrm{SO}(3)$ holonomy with irreducible non-trivial representation in $\mathrm{SO}(5)$, for well chosen weight functions. Besides, with low dimensional base spaces one can actually concentrate on a nearly infinite number of worthy examples.

We shall proceed, in a section below, to compute the Riemann curvature tensor of $\mg_{_{M,E}}$ at the zero section $O_M$. The purpose is to find in quick steps some relevant information about the Riemannian holonomy of $E$, and proceed with applications. The following observation seems to have some \textit{originality} within the extensive literature of generalized Sasaki metrics and metrics on vector bundle manifolds.

First one recalls the general theory of connections which says the holonomy group is an invariant, up to conjugation, by parallel transport over the immersed curves of class $\mathrm{C}^j,\ \forall 1\leq j\leq\infty$, inside the connection's structure Lie group and over each connected component of the manifold (cf. \cite{Joy,KobNomi}). 

For the Riemannian manifold $(E,\mg_{_{M,E}})$, we thus find some information on the holonomy group if we find it on a point of the zero section $O_M$, just because vector spaces are connected. Of course, this does not (always) prevent from having to study parallel transport of a given structure when one wishes to infer a global statement supported on the local holonomy, which, as it is well-known, is closely related to the curvature tensor.

We call \textit{local} holonomy that which is produced by the curvature tensor at a given point or subset of points. The theorem of Ambrose-Singer, well-known as a global statement, says the holonomy group of the manifold is known when the curvature tensor is known everywhere.

\subsection{Flat vector bundle}
\label{Fvbacc}

Here we study the curvature of a simple case of the metric $\mg_{_{M,E}}$. We assume $D^{^E}$ is flat and hence the setting is also a generalization to flat vector bundles of the trivial product bundle, studied by Bergery, with $f(t)=t$ as introduced in section \ref{CGMTagBm}. 

Let us take the connections ${D^{**}}=\pi^*\na^{^M}\oplus{\pi^\star}D^{^E}$ and ${\widetilde{D}}={D^{**}}+C$ defined earlier for theorem \ref{TeoremaformuladeCcoeficientes}, which now are both torsion free. On the way we are assuming two given functions $\varphi_1,\varphi_2$ of the squared-radius. The following formulas are easy to check. Of course the Levi-Civita connection is $\na^{^{M,E}}={\widetilde{D}}$ so we shall use ${\widetilde{R}}$ for the Riemannian curvature tensor of $\mg_{_{M,E}}$. The following computations may be of some originality:
 \begin{align}
  {\widetilde{R}}(X^h,Y^h)Z^h &= \label{curvaturaflathhh} \pi^*R^M(X^h,Y^h)Z^h+4r{\varphi_1'}^2e^{2\varphi_1-2\varphi_2}(X^h\wedge Y^h)(Z^h) \\
 {\widetilde{R}}(X^h,Y^h)Z^v &= 0 \\
 {\widetilde{R}}(X^h,Y^v)Z^h &= e^{2\varphi_1-2\varphi_2}\langle X^h,Z^h\rangle\bigl(4(\varphi_1''+{\varphi_1'}^2-2\varphi_1'\varphi_2')
 \xi^\flat(Y^v)\xi+2(2r\varphi_1'\varphi_2'+\varphi_1')Y^v\bigr)   \label{curvaturaflathvh} \\
 {\widetilde{R}}(X^h,Y^v)Z^v &= \bigl(4(2\varphi_1'\varphi_2' -{\varphi_1'}^2-\varphi_1'')\xi^\flat(Y^v)\xi^\flat(Z^v)-2(2r\varphi_1'\varphi_2'+ 
     \varphi_1')\langle Y^v,Z^v\rangle\bigr)X^h \label{curvaturaflathvv} \\
 {\widetilde{R}}(X^v,Y^v)Z^h &= 0 \\
 \begin{split}\label{curvaturaflatvvv}
  {\widetilde{R}}(X^v,Y^v)Z^v &= 4(\varphi_2''-{\varphi_2'}^2)\bigl(\xi^\flat(Z^v)(X^v\wedge Y^v)(\xi)-\langle X^v\wedge Y^v,\xi\wedge Z^v\rangle\xi\bigr)+ \\
  & \hspace{3cm} +4(\varphi_2'+r{\varphi_2'}^2)(X^v\wedge Y^v)(Z^v) 
 \end{split}\ .
 \end{align}
We use $(u\wedge v)z=\langle u,z\rangle v-\langle v,z\rangle u$. Constant curvature $K$ corresponds thus to ${\widetilde{R}}(u,v)z=-K(u\wedge v)z$. We also notice \eqref{curvaturaflatvvv} is a slight difference form of, but related to \eqref{curvaturaseccionalvvv}. Of course the expected Riemannian symmetries are confirmed. We also hope the formulas are useful to the reader. These computations are interesting for further study, to be carried elsewhere. E.g. there may exist Einstein metrics.

\subsection{The Riemannian curvature at the zero section}
\label{TRctatzs}

Back in the general setting let us again consider the connections ${D^{**}}$ and ${\widetilde{D}}={D^{**}}+C$. We show here the computations of the curvature in general form. Let us denote
\begin{equation}\label{curvatura1}
 R^{{\mg_{_{M,E}}}}=R^{\na^{^{M,E}}}\ .
\end{equation}
Since $\xi=0$ on $O_M$ we have $C\wedge C\:_{\mid_o}=0$ at any given point $o\in O_M$ of the zero section. It then follows by definition that
\begin{equation}\label{curvatura2}
 \begin{split}
   R^{\widetilde{D}}\ _{\mid_o} 
   &=\,R^{D^{**}}+\dx^{D^{**}}C\ \ _{\mid_o}\ .
 \end{split}
\end{equation}
Following \eqref{theLeviCivitaconnection}, the same reasons imply
\begin{equation}\label{curvatura3}
\begin{split}
R^{\mg_{_{M,E}}}\ _{\mid_o}=&\, R^{\widetilde{D}}+\dx^{\widetilde{D}}(A-\frac{1}{2}\calR)\ \  _{\mid_o}\\
=&\, R^{D^{**}}+\dx^{D^{**}}C +\dx^{D^{**}}(A-\frac{1}{2}\calR)\ \ _{\mid_o} \ .
 \end{split}
\end{equation}
Now, $X(\langle\xi,Y\rangle_{_E})=\langle X,Y\rangle_{_E}+\langle\xi, {\pi^\star}D^{^E}_XY\rangle_{_E}$ and hence, $\forall X,Y,Z,W\in TE$,
\begin{equation}\label{curvatura4aux}
 \begin{split}
  ({D^{**}}_XC_Y)Z\ \ _{\mid_o}=&\, {D^{**}}_X(C_YZ)-C_Y({D^{**}}_XZ)\ \ _{\mid_o}\\
  =&\,a\langle X,Y\rangle_{_E}Z^h+a\langle X,Z\rangle_{_E}Y^h+ c_1\langle Y,Z\rangle_{_M}X^v+\\
  &\ \ +c_2\langle Y,Z\rangle_{_E} X^v+b\langle X,Y\rangle_{_E}Z^v+b\langle X,Z\rangle_{_E}Y^v  \ .
 \end{split}
\end{equation}
 Here, $a=a_{\mid_0},\ b=b_{\mid_0}$, etc, just as for all other scalar functions --- we recall, $c_1e^{2\varphi_2}=-ae^{2\varphi_1},\ a=2\varphi_1',\ b=2\varphi_2'=-c_2$. From this last we have
\begin{equation}\label{curvatura5}
 \begin{split}
(\dx^{D^{**}}C)(X,Y)Z\ \ _{\mid_o}=&\, ({D^{**}}_XC_Y)Z-({D^{**}}_YC_X)Z-C_{[X,Y]}Z\ \ _{\mid_o}\\
  =&\,a\langle X,Z\rangle_{_E}Y^h-a\langle Y,Z\rangle_{_E}X^h+ c_1\langle Y,Z\rangle_{_M}X^v-c_1\langle X,Z\rangle_{_M}Y^v+\\
  &\ \ +2b\langle X,Z\rangle_{_E}Y^v-2b\langle Y,Z\rangle_{_E} X^v  \ .
 \end{split}
\end{equation}
Since
\begin{equation}\label{curvatura6aux}
\begin{split}
 {\widetilde{D}}_X(\calR(Y,Z))\ \ _{\mid_o}
 &\,= {\pi^\star}D^{^E}_X({\pi^\star}R^E(Y,Z)\xi)\ \ _{\mid_o}\\
 &\,= ({\pi^\star}D^{^E}_X{\pi^\star}R^E(Y,Z))\xi+{\pi^\star}R^E(Y,Z){\pi^\star}D^{^E}_X\xi\ \ _{\mid_o}\\
&\,=R^E(Y,Z)X^v
 \end{split}
\end{equation}
(with notation slightly abbreviated), we then have
\begin{equation}\label{curvatura7aux}
 ({\widetilde{D}}_X\calR_Y)Z\ \ _{\mid_o}= {\widetilde{D}}_X(\calR(Y,Z))-\calR(Y,{\widetilde{D}}_XZ)\ \ _{\mid_o}=R^E(Y,Z)X^v
\end{equation}
and
\begin{equation}\label{curvatura8aux}
 \begin{split}
\mg_{_{M,E}}(\widetilde{D}_Y(A(X,Z)),W)\ _{\mid_o}
 &=\,Y\bigl(\mg_{_{M,E}}(A(X,Z),W)\bigr) -\mg_{_{M,E}}(A(X,Z),\widetilde{D}_YW)\ _{\mid_o}\\
  &=\,\frac{1}{2}Y\bigl(e^{2\varphi_2}(\langle\calR(X,W),Z\rangle_{_E}+\langle\calR(Z,W),X\rangle_{_E})\bigr)\ _{\mid_o}\\
  &=\,\frac{1}{2}e^{2\varphi_2}\bigl(\langle R^E(X,W)Y^v,Z\rangle_{_E}+ \langle R^E(Z,W)Y^v,X\rangle_{_E}\bigr)\ .
 \end{split}
\end{equation}
Finally
\begin{equation}\label{curvatura9}
\begin{split}
   \lefteqn{  \mg_{_{M,E}}(\dx^{\widetilde{D}}(A-\frac{1}{2}\calR)(X,Y)Z,W)\ _{\mid_o}\,= }\\
 &\hspace{1cm}\,=   \mg_{_{M,E}}({\widetilde{D}}_X(A-\frac{1}{2}\calR)_Y\,Z -{\widetilde{D}}_Y(A-\frac{1}{2}\calR)_X\,Z,W)\ _{\mid_o}\\
  &\hspace{1cm}\,=  \mg_{_{M,E}}({\widetilde{D}}_X((A-\frac{1}{2}\calR)(Y,Z)) -{\widetilde{D}}_Y((A-\frac{1}{2}\calR)(X,Z)),W)\ _{\mid_o}\\
  &\hspace{1cm}\,=\frac{1}{2}e^{2\varphi_2}\bigl(\langle R^E(Y,W)X^v,Z\rangle_{_E}+\langle R^E(Z,W)X^v,Y\rangle_{_E}\\
  &\hspace{2cm} \ \ -\langle R^E(Y,Z)X^v,W\rangle_{_E}-\langle R^E(X,W)Y^v,Z\rangle_{_E}\\
  &\hspace{2cm} \ \ -\langle R^E(Z,W)Y^v,X\rangle_{_E}+\langle R^E(X,Z)Y^v,W\rangle_{_E}\bigr)\ .
\end{split}
\end{equation}
Letting $R^{\mg_{_{M,E}}}(X,Y,Z,W)=\mg_{_{M,E}}(R^{\mg_{_{M,E}}}(X,Y)Z,W)$, we may see again \eqref{curvatura3},\eqref{curvatura5},\eqref{curvatura9} and clearly deduce a set of formulas. First recall that
\begin{equation}\label{curvatura10}
 R^{D^{**}}=\pi^*R^M\oplus{\pi^\star}R^E\ .
\end{equation}
\begin{teo}\label{curvatura11}
 Let $x\in M,\ o\in O_M\subset E$ with $\pi(o)=x$. Then at point $o$
 \begin{align}
 %\mbox{for}\ X,Y,Z,W\ \mbox{horizontal}, \qquad 
 R^{\mg_{_{M,E}}}_o(X^h,Y^h,Z^h,W^h) &  = e^{2\varphi_1}\langle\pi^*R^M_x(X^h,Y^h)Z^h,W^h\rangle_{_M} \\ 
 % \mbox{for}\ X,Y,Z\ \mbox{hor.},\ W\ \mbox{vertical}, \qquad 
 R^{\mg_{_{M,E}}}_o(X^h,Y^h,Z^h,W^v) & =0  \\
% \mbox{for}\ X,Y\ \mbox{hor.} , \ Z,W\ \mbox{vert.}, \qquad  
R^{\mg_{_{M,E}}}_o(X^h,Y^h,Z^v,W^v) &=
       e^{2\varphi_2}\langle{\pi^\star}R^E_x(X^h,Y^h)Z^v,W^v\rangle_{_E} \\
% \mbox{for}\ X,Z\ \mbox{hor.},\ Y,W\ \mbox{vert.}, \qquad \hspace{8cm} \nonumber \\
 R^{\mg_{_{M,E}}}_o(X^h,Y^v,Z^h,W^v) &= ae^{2\varphi_1}\langle X^h,Z^h\rangle_{_M}
 \langle Y^v,W^v\rangle_{_E} +\frac{1}{2}e^{2\varphi_2} \langle{\pi^\star}R^E_x(X^h,Z^h)Y^v,W^v\rangle_{_E}   \\ 
 % \mbox{for}\ X,Y\ \mbox{vert.} , \ Z,W\ \mbox{hor.}, \qquad 
  R^{\mg_{_{M,E}}}_o(X^v,Y^v,Z^h,W^h) &= e^{2\varphi_2}\langle{\pi^\star}R^E_x(Z^h,W^h)X^v,Y^v\rangle_{_E} \\
%  \mbox{... }, \quad \quad  
R^{\mg_{_{M,E}}}_o(X^v,Y^v,Z^v,W^h) &=0 \\
  R^{\mg_{_{M,E}}}_o(X^h,Y^v,Z^v,W^v) &=0  \\
  R^{\mg_{_{M,E}}}_o(X^v,Y^v,Z^v,W^v) &=
 -2be^{2\varphi_2}(\langle X^v,W^v\rangle_{_E}\langle Y^v,Z^v\rangle_{_E}
 -\langle X^v,Z^v\rangle_ {_E}\langle Y^v,W^v\rangle_{_E})\ .
 \end{align}
Recall, here, $a=a_{\mid_0},\ b=b_{\mid_0}$, etc, as well as with all other scalar functions.
\end{teo}
Theorem \ref{curvatura11} can also be checked case by case, as it is done in some references for the curvature of the Sasaki and generalized Sasaki metrics on $E=TM$, cf. \cite{Kow1,Seki}. Now recall $E$ has rank $k$ and suppose $M$ has dimension $m$. In the following result it is a remarkable surprise that the curvature of $D^{^E}$ has completely disappeared.
\begin{teo}\label{curvaturasdeRiccieescalar}
 The Ricci tensor $\ric^{\mg_{_{M,E}}}(X,Y)=\tr{R^{\mg_{_{M,E}}}(\ ,X)Y}$ and the scalar curvature $\Scal^{\mg_{_{M,E}}}=\mathrm{tr}_{\mg_{_{M,E}}}\ric^{\mg_{_{M,E}}}$ satisfy ($a=a_{\mid_0},\ b=b_{\mid_0}$, as well as with all other scalar functions):
 \begin{align}
  \ric^{\mg_{_{M,E}}}_o(X^h,W^h)&=\ric^M_x(X^h,W^h)-ak e^{2(\varphi_1-\varphi_2)}\langle X^h,W^h\rangle_{_M}\\
  \ric^{\mg_{_{M,E}}}_o(X^h,W^v)&=0\\
  \ric^{\mg_{_{M,E}}}_o(X^v,W^v)&=(2b(1-k)-am)\langle X,W\rangle_{_E}
\end{align}
and also at $o$
\begin{align}
  \Scal^{\mg_{_{M,E}}}_o&= e^{-2\varphi_1}\Scal^M_x+e^{-2\varphi_2}(2bk(1-k)-2akm)\ .
 \end{align}
\end{teo}
\begin{coro}\label{curvaturasdeRiccieescalarcorolario}
 If the Riemannian manifold $E$ is Einstein, hence satisfying $\ric^{\mg_{_{M,E}}}=\lambda^{^{E}}\mg_{_{M,E}}$, then $M$ is Einstein say with Einstein constant $\lambda^{^M}$ and at $o$ we have
 \begin{equation}
  \lambda^{^M} e^{2\varphi_2-2\varphi_1}+a(m-k)+2b(k-1)=0\ .
 \end{equation}
Moreover
\begin{equation}\label{curvaturasdeRiccieescalarEinstein}
\lambda^{^E} = (2b(1-k)-am)e^{-2\varphi_2} 
=\lambda^{^M}e^{-2\varphi_1}-ake^{-2\varphi_2} \ .
\end{equation}
\end{coro}

The holonomy equations given by theorem \ref{curvatura11} generate a Lie subalgebra of the orthogonal Lie algebra of $T_oE$, in matrix form respecting the canonical decomposition. There are obviously three kinds of operators $R^{\mg_{_{M,E}}}_o(X,Y)$:
\begin{equation}\label{holonomia1}
 \left[\begin{array}{cc} e^{2\varphi_1}R^M(X^h,Y^h)&0 \\ 0 &e^{2\varphi_2}R^E(X^h,Y^h)
 \end{array}\right]
\end{equation}
\begin{equation}\label{holonomia2}
 \left[\begin{array}{cc} 0&-B \\B^\dag& 0
 \end{array}\right]\quad\mbox{with}\quad
    B(X^h,Y^v)=2\varphi_1'e^{2\varphi_1}(X^h)^\flat\otimes(Y^v)^\flat+\frac{1}{2}e^{2\varphi_2}\langle R^E(X^h,\ )Y^v,\ \rangle_{_E}\\
\end{equation}
and
\begin{equation}\label{holonomia3}
           \left[\begin{array}{cc}
 e^{2\varphi_2}\langle R^E(\ ,\ )X^v,Y^v\rangle_{_E} &0\\ 0& 4\varphi_2'e^{2\varphi_2}(X^v)^\flat\wedge(Y^v)^\flat  \end{array}\right]
\end{equation}
where $B^\dag$ is the adjoint endomorphism of $B$ with respect to the product metric (not the weighted). By the celebrated Ambrose-Singer theorem these endomorphisms generate the local holonomy algebra. That is, a Lie subalgebra of the Lie algebra of the Riemannian holonomy group of $E$. It is also the moment to recall the last of the two observations in section \ref{Twoobserv}.

\subsection{The flat connection again}

Suppose $D^{^E}$ is a flat connection on $E\lrr M$ with the dimensions $m=\dim M,\ k=\mathrm{rk}\,E$. At a point $o\in O_M$ we find the Riemannian curvature of $\mg_{_{M,E}}$ from (\ref{holonomia1}---\ref{holonomia3}). We get in the most general case, i.e. when both $\varphi_1'(0),\varphi_2'(0)$ do not vanish, the 
three types of endomorphisms in $\sol(T_oE,\mg_{_{M,E}})\simeq\Lambda^2\R^{m+k}= \Lambda^2\R^{m}\oplus\p\oplus\Lambda^2\R^{k}$:
\begin{equation}\label{holonomiaflat1}
 \left[\begin{array}{cc}  R^M & 0 \\ 0 & 0    \end{array}\right]\qquad
 \left[\begin{array}{cc} 0 & -{E_i^\alpha}\\ (E_i^\alpha)^\dag  & 0 \end{array}\right]\qquad
 \left[\begin{array}{cc} 0 & 0 \\ 0 &  e^\alpha\wedge e^\beta    \end{array}\right]\qquad\forall i,j,\alpha,\beta.
\end{equation}
The matrices $E_i^\alpha=[\delta_i^p\delta_\alpha^\beta]_{p\beta}$ and those in the middle generate the subspace $\p$ of dimension $mk$. Since $[\p,\p]=\sol(m)\oplus\sol(k)$, we find the first part of the following result.
\begin{prop}\label{holonomiaflat2} Let $\hol^{\mg_{_{M,E}}}$ denote the whole holonomy Lie algebra.\\
(i) If $\varphi_1'(0)\neq0$, then $\hol^{\mg_{_{M,E}}}=\sol(m+k)$. \\
(ii) If $\varphi_1'(0)=0\neq\varphi_2'(0)$, then $\hol^{\mg_{_{M,E}}}\supseteq\hol^M\oplus\sol(k)$.\\
(iii)  $\hol^{\mg_{_{M,E}}}\supseteq\hol^M$, with equality if both $\varphi_1,\varphi_2$ are constant.
\end{prop}
The second and third assertions follow from (\ref{curvaturaflathhh}---\ref{curvaturaflatvvv}).

We see clearly now that we cannot eliminate the chance of \eqref{curvaturaflatvvv} producing a smaller vertical holonomy subalgebra in the event of just $\varphi_2'(0)=0$. But that is a matter for the strict study of spherically symmetric metrics on euclidean space.

\section{Applications}

\subsection{Hermitian tangent bundle with generalized Sasaki metric}

We now start looking for some applications of the theory above with a particular case of a well-known result. Given any Riemannian manifold $M$, the generalized Sasaki almost Hermitian structure consists of the $\mg_{_{M,E}}$-compatible almost complex structure $J^{\mathrm{S,\psi}}$ on the manifold $E=TM$ defined by
\begin{equation}
 J^{\mathrm{S,\psi}}=e^{-\psi}B-e^{\psi}B^\dag
\end{equation}
where $\psi=\varphi_2-\varphi_1$ and the endomorphism $B:TTM\lrr TTM$ is defined by $BZ^h=Z^v,\ BZ^v=0$. The map $B$ is a well-defined structure on $TM$ which cannot be reproduced on other vector bundles. It has proven quite useful in other studies, cf. \cite{Alb2008,Alb2012,Alb2014a}\footnote{We remark that the structures of generalized Sasaki type with weight functions dependent of the base point $x\in M$, rather than the squared-radius $r$, have been studied by the author in \cite{Alb2012}.}. The almost complex structure $J^{\mathrm{S,\psi}}$ clearly generalizes the case $\varphi_1=\varphi_2=0$, which we call the \textit{canonical case} and is due to S.~Sasaki. The almost Hermitian structure is trivial to check. Notice here the adjoint $B^\dag$ is with respect to the canonical case. Defining also $\overline{\psi}=\varphi_2+\varphi_1$ it follows\footnote{One may say the close relations between metric and complex structures start with the twist $\varphi_1,\varphi_2\mapsto\psi,\overline{\psi}$.} that the associated symplectic 2-form $\omega^{\mathrm{S,\overline{\psi}}}=J^{\mathrm{S,\psi}}\lrcorner\mg_{_{M,E}}$ satisfies
\begin{equation}
  \omega^{\mathrm{S,\overline{\psi}}}=e^{\overline{\psi}}\omega^{\mathrm{S,0}}\ .
\end{equation}
\begin{prop}\label{propsymplec}
In case $\dim M=m>1$, the 2-form $\omega^{\mathrm{S,\overline{\psi}}}$ is symplectic if and only if $\overline{\psi}$ is a constant. 
\end{prop}
\begin{proof}
 This is easy after seeing that $\omega^{\mathrm{S,0}}$ is always an exact form (cf. the beginning of section \ref{CGMTagBm}). Then  $\dx\overline{\psi}\wedge\omega^{\mathrm{S,0}}=0$ implies $\overline{\psi}$ is constant, by trivial reasons. Here we shall restart with another proof, however, by showing that the canonical 2-form is closed. Indeed, the connection $D^{**}$, for which $B$ and the 2-form are clearly parallel, has torsion $T^{{D^{**}}}(X,Y)=\pi^\star R^E(X,Y)\xi$ due to \eqref{equacoesdatorsao}. A well-known formula says
 \begin{eqnarray*} 
 \dx\omega^{\mathrm{S,0}}(X,Y,Z)&=& \omega^{\mathrm{S,0}}(T^{{D^{**}}}(X,Y),Z)+\omega^{\mathrm{S,0}}(T^{{D^{**}}}(Y,Z),X)+\omega^{\mathrm{S,0}}(T^{{D^{**}}}(Z,X),Y) \ .
 \end{eqnarray*}
 The result follows by Bianchi identity since $\na^{^M}$ is torsion-free.
\end{proof}
On the other hand $J^{\mathrm{S,\psi}}$ is integrable if $R^M=0$ and $\psi$ is a constant. In sum, $TM$ is both complex analytic and symplectic, i.e. a K\"ahlerian manifold, if and only if $M$ is flat and $\varphi_1,\varphi_2$ are constants. 

It is however possible to find subspaces of $TM$ which are complex analytic for other classes of $M$. The equation of integrability led us to that discovery, in \cite{Alb2014c}. 

In any dimension, the unique solutions are the tangent disk bundles $D_{r_0}M\lrr M$ of real space forms of sectional curvature $\kappa$ with any squared-radius $r_0\in\R^+$ and metric satisfying $\psi$ constant. For a complete metric, $r_0=+\infty$ when $\kappa>0$.  Such disk bundles are, moreover, K\"ahlerian; by the proposition above we take $\overline{\psi}=0$. The family of metrics, with a controlled parameter $c_1$, is given by
\begin{equation}
 \mg_{_{M,DM}}=\sqrt{c_1+\kappa r}\,\pi^*g_{_M}\oplus\frac{1}{\sqrt{c_1+\kappa r}}\,{\pi^\star}g_{_{TM}}
\end{equation}
where $r_0=-c_1/\kappa$ if $\kappa<0$ or $+\infty$ otherwise. We refer the reader to \cite{Alb2014c} for details. Again a local holonomy unitary subgroup is found through the results of section \ref{TRctatzs}. In general, regarding the Hermitian or unitary holonomy of $\mg_{_{M,TM}}$ and the corresponding Levi-Civita connection $\na^{^{M,TM}}$, the zero section $O_M$ can only tell us about the whole manifold geometry when we have also $J^{\mathrm{S,\psi}}$ parallel. 

The flat base space has $TM$ with Riemannian holonomy group $\SO(2m)$ if $\varphi_1'(0)\neq0$. The holonomy is trivial if both weight functions are constant, cf. proposition \ref{holonomiaflat2}.

\subsection{Metrics with $\gdois$ holonomy}

In this section $M$ denotes an oriented Riemannian 4-manifold.

In \cite{Alb2014b} the author studied some generalizations of the metrics of Bryant-Salamon (\cite{BrySal}) on the vector bundle $E=\Lambda^2_\pm T^*M$ of self-dual and anti-self-dual 2-forms on $M$. Recall those metrics were originally found on $\Lambda_-^2$ for positive scalar curvature self-dual Einstein manifolds, essentially $S^4$ and $\C\Proj^2$. A change of orientation is not a restraint, but in working with $\Lambda^2_+$ one finds a perfect mirror construction for negative scalar curvature, and thus finds an unknown number of new examples of Riemannian 7-manifolds with $\gdois$ holonomy. In particular for the Einstein base $M=\calH^4$ and $\calH^2_\C$, respectively, the real and complex hyperbolic spaces.

Pure identity is not perfection and hence, when trying to find the holonomy subgroup of $\gdois$, the Lie theory for those \textit{new} symmetric spaces does not apply. As proved by the author in \cite{Alb2014b}, the arguments of Bryant-Salamon cannot be reproduced for $\Scal^M<0$. Our main purpose here is therefore to give a general proof that, for certain spaces, including the original, the holonomy groups are the whole $\gdois$ Lie group. The study includes the scalar flat base case, which yields a different conclusion.

Similar metrics on vector bundle manifolds with $\gdois$ holonomy were also found by  G. W. Gibbons, D. N. Page and C. N. Pope in \cite{GibbPagePope}. Indeed this reference deals with bundles over $S^4$ and $\C\Proj^2$, but does not see the $\Scal^M\leq0$ cases. To the best of our knowledge, the same is true for all recent research in the field.

Let us recall the manifold $E=\Lambda^2_\pm T^*M$ and its metric $\mg_\phi$ in general. We shall not be so focused on the $\gdois$ structure, the 3-form $\phi$, which determines the metric. 

Given an oriented orthonormal frame $\{e^4,e^5,e^6,e^7\}$ of $T^*M$ on an open subset, we have a frame on $E$ on the same open subset defined by\footnote{Everywhere possible, we omit the $\pm$ which is attached to each 2-form and vector bundle.}:
\begin{equation}\label{selfdualformsandantiselfdualforms}
 e^1=e^{45}\pm e^{67}\ ,\qquad e^2=e^{46}\mp e^{57}\ ,\qquad e^3=e^{47}\pm e^{56}\ .
\end{equation}
One may now carefully check that the metric $g_{_E}$ on the vector bundle is the unique such that $\{e^1,e^2,e^3\}$ constitute an orthonormal frame\footnote{The $e^i,\ i=1,2,3$, have norm 2 for the usual metric on 2-forms, but indeed it is $\frac{1}{2}$ of this that is used in \cite{Alb2014b}. In particular, the notation here for $r$ refers to the \textit{half} squared radius $r$ mentioned there.}. 

The metrics we are interested, as explained above, come in pairs. Let us assume that, when $\Lambda^2_-$ is considered, then the base manifold $M$ is Einstein and self-dual (i.e. has vanishing anti-self dual Weyl tensor $W_-=0$). And when we refer to $\Lambda^2_+$, the base manifold is Einstein anti-self-dual ($W_+=0$). The vector bundles $E$ inherit a metric connection $\na$ from the Levi-Civita connection $\na^M$, which of course commutes with the Hodge star operator.

Finally the desired metric on $E$ is a spherically symmetric metric $\mg_{_{M,E}}$ with certain weight functions, cf. \cite{Alb2014b,BrySal}:
\begin{equation}
 \mg_\phi=\mg_{_{M,E}}=\sqrt{2\ctildezero^2sr+\ctildeum}\,\pi^*g_{_M}\oplus 
 \frac{\ctildezero^2}{\sqrt{2\ctildezero^2sr+\ctildeum}}\,\pi^\star g_{_E}
\end{equation}
where $\ctildezero,\ctildeum>0$ are constants and $s=\frac{1}{12}\Scal^M$. There are clearly two degrees of freedom in this metric, concerning the conformal changes on $M$ and on the fibres of $E$, which are easy to normalize. Nevertheless, we shall keep the two constant indeterminants $\ctildezero,\ctildeum$ until the end.

Recall we already know the metric $\mg_\phi$ has holonomy in $\gdois$ (i.e. $\phi$ is parallel for the Levi-Civita connection itself induces). Implicitly there is the metric connection $\na$ for $g_{_E}$. One close look at the defining equations in \cite{Alb2014b} will prove the horizontal subspace is defined by \eqref{canonicaldecompo},\eqref{projeccaoprincipal}.

Now we must abide to one more little detail, because for $\Scal^M<0$ of course the definitions are not consistent. This is easily solved by restricting the study to the open disk bundle of radius $\sqrt{r_0}$, where $r_0=-\ctildeum/2s$,
\begin{equation}
 D_{r_0,\pm}M=\{ e\in E:\ \|e\|_{_E}^2<r_0\}\ .
\end{equation}

The spherically symmetric metric weights are given by
\begin{equation}
 \varphi_1(r)=\frac{1}{4}\log(2\ctildezero^2 sr+\ctildeum)\qquad\qquad
\varphi_2(r)=-\frac{1}{4}\log(2\ctildezero^2 sr+\ctildeum)+\log\ctildezero\ .
\end{equation}
Then
\begin{equation}
 e^{2\varphi_1(0)}=\ctildeum^\frac{1}{2}\qquad\qquad\quad
 e^{2\varphi_2(0)}=\ctildezero^2\ctildeum^{-\frac{1}{2}}
\end{equation}
\begin{equation}
 \varphi_1'(r)=\frac{\ctildezero^2s}{2(2\ctildezero^2sr+\ctildeum)}\qquad\qquad
\varphi_2'(r)=-\frac{\ctildezero^2s}{2(2\ctildezero^2 sr+\ctildeum)}
\end{equation}
and
\begin{equation}
 \varphi_1'(0)=\frac{\ctildezero^2s}{2\ctildeum}=-\varphi_2'(0)\ .
\end{equation}
As it is becoming apparent, we shall indeed use duely the zero section. Regarding the famous coefficients from theorem \ref{TeoremaformuladeCcoeficientes}, defined by $a=2\varphi_1',\ b=2\varphi_2'=-c_2,\ c_1=-2\varphi_1'e^{2\varphi_1-2\varphi_2}$, we find at 0
\begin{equation}
   a=-b=c_2=\frac{\ctildezero^2s}{\ctildeum}\qquad c_1=-s \ .
\end{equation}

It is known the special integrable geometry of $\gdois$ is Ricci flat. In our case this may be confirmed by corollary \ref{curvaturasdeRiccieescalarcorolario}, in particular through formula \eqref{curvaturasdeRiccieescalarEinstein}. Indeed, we have both
\begin{equation}
 2b(1-k)-am=-4(b+a)=0
\end{equation}
and
\begin{equation}
 \lambda^{^M}e^{-2\varphi_1}-ake^{-2\varphi_2}=3s\ctildeum^{-\frac{1}{2}}
           -3\frac{\ctildezero^2s}{\ctildeum}\ctildezero^{-2}\ctildeum^{\frac{1}{2}}=0\ .
\end{equation}
We now write our main result for the metrics $\mg_\phi$ on $\Lambda^2_-T^*M$ if $s>0$ and $D_{r_0,+}M$ if $s<0$.
\begin{teo}\label{AholonomiaG2paratodoosnaonulo}
 For $s\neq0$, the holonomy group of $\mg_\phi$  is the Lie group $\gdois$.
\end{teo}
\begin{proof}
 We must recall the decomposition of the curvature tensor of 4-manifolds under the Lie algebra isomorphism $\sol(4)=\sol(3)\oplus\sol(3)$. A reference is \cite{Besse}; we recall the lines of \cite{Alb2014b}. The symmetric operator on 2-forms defined by 
  \begin{equation}\label{curvaturaFourManifolds}
  \langle {\cal R}(e_\alpha\wedge e_\beta),e_\gamma\wedge e_\delta\rangle_{_M}=-\langle R^M(e_\alpha,e_\beta)e_\gamma,e_\delta\rangle_{_M}=R^M_{\alpha\beta\gamma\delta} 
     \end{equation}
 gives rise to an irreducible decomposition respecting $\Lambda_+\oplus\Lambda_-$
 \begin{equation}
 {\cal R}=\left[\begin{array}{cc}
          W_++s1_3 & \ric_0 \\ {\ric_0}^\dag & W_-+s1_3     \end{array}\right]\ .     
 \end{equation}
 The Weyl tensor is $W=W_++W_-$ where $W_\pm$ are traceless. The map $\ric_0$ is the traceless part of $\ric^{g_M}$. It follows $s=\frac{1}{12}\Scal^M$.
 
 Now the curvature of the vector bundle $E$ is given in the frame $p=(e^1,e^2,e^3)$ by
 \[ R^Ep=p\rho\qquad\quad\mbox{where}\qquad\rho= \left[\begin{array}{ccc}
              0 & -\rho^3 &\rho^2 \\   \rho^3& 0 &-\rho^1 \\ -\rho^2 &\rho^1 &0
              \end{array} \right]   \ .      \]
 In other words
 \begin{equation}\label{curvaturaselfdualeantiselfdual}
 R^Ee^i=\rho^ke^j-\rho^je^k,\  \quad\forall\ \mbox{cycle}\ (ijk)=(123)\ .
 \end{equation}
 We may write again $\rho$, more precisely each $\rho^i_+$ and $\rho^i_-$, as a linear combination of the self-dual and anti-self dual 2-forms. Taking a dual frame of the $e^4,e^5,e^6,e^7$, we get respective 2-vectors $e_1,e_2,e_3$, which verify $e^i_\pm(e_{\pm,j})=2\delta^i_j,\ e^i_\pm(e_{\mp,j})=0$,\ $\forall i,j=1,2,3$. Careful computations, cf. \cite{Alb2014b}, yield:
 \[ \rho_+^i(e_{+,j})=-{\cal R}_{ij}\qquad
    \rho_\pm^i(e_{\mp,j})=\mp{\cal R}_{i\bar{j}}  \qquad 
    \rho_-^i(e_{-,j})=+{\cal R}_{\bar{i}\bar{j}}  \qquad\forall i,j=1,2,3 .  \]
 Notice, for instance, ${\cal R}_{ij}=\langle{\cal R}e_i,e_j\rangle_{_M}$ is the extension through \eqref{curvaturaFourManifolds} with the metric $g_{_M}$. If $M$ is Einstein, equivalently if $\ric_0=0$, then $s$ is a constant.
 
 As said above, $M$ is self-dual if $W=W_+$. Self-dual and Einstein is the same as
\begin{equation}\label{selfdualcurvatureequation}
 W_-=0\qquad\Longleftrightarrow\qquad\ \rho^i_-=se^i_-\ ,\quad \forall i=1,2,3\ .
\end{equation}
 Anti-self-duality corresponds to $W=W_-$. Together with the Einstein condition, that implies $\rho^i_+=-se^i_+$. All together, the hypothesis are henceforth that $\rho^i=\mp se^i$.
 
 Now we are ready for the computation of the dimension of the holonomy Lie algebra. Indeed we just have to prove $\dim\hol^{\mg_{_{M,E}}}=14$, since 14 is the dimension of $\gdois$. Theorem \ref{curvatura11} leads to the answer. By formulas (\ref{holonomia1}--\ref{holonomia3}) and subsequent observations, the holonomy is generated by those $\sol(7)$-type matrices. Let us recall, in the corresponding order and introducing the weights,
 \begin{equation*}
 R^{\mg_\phi}(X^h,Y^h)=\left[\begin{array}{cc} \ctildeum^\frac{1}{2}R^M(X^h,Y^h)&0 \\ 0 &\ctildezero^2\ctildeum^{-\frac{1}{2}}R^E(X^h,Y^h)
 \end{array}\right]
\end{equation*}
\begin{equation*}
 R^{\mg_\phi}(X^h,Y^v)=\left[\begin{array}{cc} 0&-B \\B^\dag& 0
 \end{array}\right]
 \end{equation*}
 with $B(X^h,Y^v)=a\ctildeum^\frac{1}{2}(X^h)^\flat\otimes(Y^v)^\flat+ 
 \frac{1}{2}\ctildezero^2\ctildeum^{-\frac{1}{2}}\langle R^E(X^h,\ )Y^v,\ \rangle_{_E}$
and
\begin{equation*}
    R^{\mg_\phi}(X^v,Y^v)=  \left[\begin{array}{cc}
 \ctildezero^2\ctildeum^{-\frac{1}{2}}\langle R^E(\ ,\ )X^v,Y^v\rangle_{_E} &0\\ 0& 2b\ctildezero^2\ctildeum^{-\frac{1}{2}}(X^v)^\flat\wedge(Y^v)^\flat  \end{array}\right]
\end{equation*}
where $B^\dag$ is the adjoint endomorphism of $B$. 
 
Notice we may consider the horizontal lift of $e_i$ as well as the vertical lift of $e^i$, which we have denoted by $\pi^\star e^i$. Recall $\langle e^i,e^j\rangle_{_E}=\frac{1}{2}\langle e^i,e^j\rangle_{_M}=\delta_{ij},\ \forall i,j=1,2,3$. We then conclude various identities on the manifold $E$. First, in coherence with the above,
\[ \frac{1}{2}\langle R^M(e_k),e_i\rangle_{_M}=-\frac{1}{2}{\cal R}_{ki}= 
   \pm\frac{1}{2}\rho^k_\pm(e_i)=-s\delta_{ik} \]
and hence $R^M(e_k)=-se^k$ (the horizontal lift, the pullback).
Second, from \eqref{curvaturaselfdualeantiselfdual} and in positive order $(ijk)$,
\[ \langle R^E(e_k)\pi^\star e^i,\pi^\star e^j\rangle_{_E}=\rho^k(e_k)=\mp2s=
     \mp2s(\pi^\star e^i\wedge\pi^\star e^j)(e_i,e_j)   \ . \]
Finally, the orthogonal maps $R^{\mg_\phi}(e_k^h)$ and $R^{\mg_\phi}(\pi^\star e^i,\pi^\star e^j)$ are equal, respectively, to
\begin{equation*}
\left[\begin{array}{cc} -\ctildeum^\frac{1}{2}se^k&0 \\ 0 &\mp2s\ctildezero^2\ctildeum^{-\frac{1}{2}}\pi^\star e^i\wedge\pi^\star e^j
 \end{array}\right]\quad\mbox{and}\quad\left[\begin{array}{cc}
 \mp\ctildezero^2\ctildeum^{-\frac{1}{2}}se^k &0\\ 0& -2s\ctildezero^4\ctildeum^{-\frac{3}{2}}\pi^\star e^i\wedge\pi^\star e^j  \end{array}\right]\ .
\end{equation*}
 In sum,
 \[ \pm\frac{\ctildezero^2}{\ctildeum}R^{\mg_\phi}(e_k^h)=R^{\mg_\phi}(\pi^\star e^i,\pi^\star e^j)\  \]
 and we have proved all these 6 maps generate a 3-dimensional subspace. 
 
 Certainly there is another 3-dimensional subspace of maps, non-vanishing just in the $4\times 4$-square, generated by the $R^{\mg_\phi}(e_{\bar{k}}^h)$. They refer to $W_\mp+s1_3$ and do not vanish because $s\neq0$. (And notice \textit{sometimes} $W_+$ or $W_-$ do not vanish either.)
 
 We are left to prove the $R^{\mg_\phi}(X^h,Y^v)$ generate an 8-dimensional subspace. Let us take any $\alpha=4,5,6,7$. Letting $\theta^i=\langle\pi^\star e^i,\ \rangle_{_E}$, we have:
 \begin{align*}
 \lefteqn{ R^{\mg_\phi}(e_\alpha,\pi^\star e^i,Z^h,W^v)=} \\
 & = a\ctildeum^\frac{1}{2}\langle e_\alpha,Z^h\rangle_{_M}\langle\pi^\star e^i,W^v\rangle_{_E}  +\frac{1}{2}\ctildezero^2\ctildeum^{-\frac{1}{2}} 
  \langle R^E(e_\alpha,Z^h)\pi^\star e^i,W^v\rangle_{_E} \\
  &= \frac{\ctildezero^2s}{2\ctildeum^{\frac{1}{2}}}\bigl(2e^\alpha(Z^h)\langle \pi^\star e^i,W^v\rangle_{_E}\mp e^k(e_\alpha,Z^h)\langle\pi^\star e^j,W^v\rangle_{_E}\pm e^j(e_\alpha,Z^h)\langle\pi^\star  e^k,W^v\rangle_{_E}\bigr)\ .
 \end{align*}
 Hence
 \[  R^{\mg_\phi}(e_\alpha,\pi^\star e^i)=\frac{\ctildezero^2s}{2\ctildeum^{\frac{1}{2}}}
      \bigl(2e^\alpha\wedge \theta^i\mp e_\alpha\lrcorner e^k\wedge\theta^j\pm e_\alpha\lrcorner e^j\wedge\theta^k\bigr) \ .  \]
  Computing case by case we get four clearly linearly independent families of three \textit{similar} 2-forms. Writting $V_{\alpha i}=e^\alpha\wedge\theta^i$, we get for instance
 \[ \begin{cases}
       R^{\mg_\phi}(e_4,\pi^\star e^1)=2V_{41}\mp V_{72}\pm V_{63} \\
       R^{\mg_\phi}(e_7,\pi^\star e^2)=2V_{72}\mp V_{41}+V_{63} \\
       R^{\mg_\phi}(e_6,\pi^\star e^3)=2V_{63}\pm V_{41}+V_{72}
    \end{cases}  \ . \]
 These forms are linearly dependent. In fact, all the following matrices, corresponding to the four families, have rank 2:
 \[ \begin{array}{ccc}    
       \left[\begin{array}{ccc}
           2 & \mp 1 & \pm 1 \\ \mp1 & 2& 1 \\ \pm1 & 1 & 2 
        \end{array}\right]  & &  \left[\begin{array}{ccc}
           2 & -1 &  -1 \\ -1 & 2 & -1 \\ -1 & -1 & 2 
     \end{array}\right]  \\
        & & \\
     \left[\begin{array}{ccc}
           2 & 1 &  \mp1 \\ 1 & 2& \pm1 \\ \mp1 & \pm1 & 2 
     \end{array}\right] & & \left[\begin{array}{ccc}
           2 & \pm1 &  1 \\ \pm1 & 2& \mp1 \\ 1 & \mp1 & 2 
         \end{array}\right]     \end{array}       \]
 and therefore the curvature generates a subspace of dimension 8.
\end{proof}

The case $s=0$ implies constant $\varphi_1,\varphi_2$. Let us finally solve this case. Recall the famous K3 surfaces are complex surfaces; with the Calabi-Yau metric, they have a preferred orientation, are Ricci-flat and anti-self-dual. With holonomy equal to $\SU(2)$. By a theorem of C. Lebrun, K3 surfaces and quotients of the 4-torus by finite groups give us all the compact spin Ricci-flat K\"ahler surfaces and hence anti-self-dual 4-manifolds. Anti-self-duality happens necessarily with every scalar flat K\"ahler surface.

The flat case being trivial, we follow on to another foreseeable result.
\begin{teo}
 For any K3 surface $M$, the $\gdois$ metrics on $\Lambda_+^2T^*M$ have holonomy $\SU(2)$. 
\end{teo}
\begin{proof}
 Of course we recur to the global formulas (\ref{curvaturaflathhh}--\ref{curvaturaflatvvv}) because $E$ is flat for $s=0$, as we have seen. 
\end{proof}

Let us stress we have completed in theorem \ref{AholonomiaG2paratodoosnaonulo} the proof of \cite[Theorem 2.4.]{Alb2014b}. Now we are completely sure the spaces 
\begin{equation}
 D_{r_0,\pm}{\cal H}^4\qquad\mbox{and}\qquad D_{r_0,+}\calH_\C^2\ ,
\end{equation}
with the metric $\mg_\phi$, have $\gdois$ holonomy. Let us see a topological proof for the Bryant-Salamon metrics. This will give the third independent proof, again suitable only for the positive $\Scal^M$ cases.
\begin{prop}
 The $\gdois$ metric on $\Lambda^2_-T^*S^4$ has holonomy equal to $\gdois$.
\end{prop}
\begin{proof}
 A theorem in \cite{Bryant1} assures that if the holonomy group is contained in $\gdois$, which is the case, and the metric does not admit parallel vector fields, then the subgroup coincides with the whole group. Now if $E$ had a parallel vector field for the $\gdois$ metric, then this would restrict on the zero section $O_M$ to the sum of a parallel vector field and a parallel section, by theorem \ref{parallelvectorfieldsonEandM}. These fields would have constant norm. But it is well-known that $S^4$ does not have non-vanishing vector fields, nor it admits a non-degenerate 2-form field (an almost-complex structure). Of course every self or anti-self-dual 2-form is a non-degenerate 2-form. 
\end{proof}
Analogous result follows for $\C\Proj^2$, because it does not admit a non-vanishing vector field, nor a K\"ahler structure compatible with the Fubini-Study metric and inducing the reversed orientation. It is well-known that $\overline{\C\Proj}^2$ is not even a complex manifold.

\vspace{1.5cm}

%\medskip

\

\textbf{R. Albuquerque}

{\texttt{\textbf{rpa@uevora.pt}}}

Centro de Investiga\c c\~ao em Mate\-m\'a\-ti\-ca e Aplica\c c\~oes

Rua Rom\~ao Ramalho, 59

671-7000 \'Evora, Portugal
\vspace{1mm}\\
\begin{small}The research leading to these results has received funding from the People Programme (Marie Curie Actions) of the European Union's Seventh Framework Programme (FP7/2007-2013) under REA grant agreement n${}^\circ$ PIEF-GA-2012-332209.\end{small}

\end{document}